\documentclass[10pt]{article}
\usepackage{epsf,textcomp}
 \usepackage[active]{srcltx}
\usepackage[usenames]{color}
\usepackage{psfrag,bbm,
graphicx, harvard,mathrsfs,
epsf, euscript,amsfonts,amsmath,latexsym,amssymb,latexsym,amssymb,
latexsym,amsthm}
\usepackage{changes}
\usepackage[final]{hyperref}
\usepackage[toc,page]{appendix}
\newtheorem{lemma}{Lemma}
\newtheorem{theorem}{Theorem}

\newtheorem{remark}{Remark}

\newtheorem{corollary}{Corollary}
\usepackage{fullpage}

\newcommand{\rd}{{\rm d}}
\newcommand{\cA}{{\cal A}}

\newcommand{\cC}{{\cal C}}
\newcommand{\cD}{{\cal D}}

\newcommand{\cF}{{\cal F}}
\newcommand{\cG}{{\cal G}}

\newcommand{\cK}{{\cal K}}

\newcommand{\cR}{{\cal R}}

\newcommand{\bE}{\mathbb E}

\newcommand{\bL}{{\mathbb L}}

\newcommand{\bP}{{\mathbb P}}

\newcommand{\bR}{{\mathbb R}}

\newcommand{\sF}{{\mathscr F}}

\newcommand{\sT}{{\mathscr T}}

\newcommand{\sX}{{\mathscr X}}

\begin{document}
\numberwithin{equation}{section}

\title{Optimal single threshold stopping rules and sharp prophet
inequalities
}
\author{
A. Goldenshluger\thanks{Department of Statistics, University of Haifa, Haifa 31905, Israel.  e-mail:
{\em goldensh@stat.haifa.ac.il}. }
\and Y. Malinovsky\thanks{Department of Mathematics and Statistics, University of Maryland,
Baltimore County, Baltimore, MD 21250, USA.
e-mail: {\em yaakovm@umbc.edu}}
\and
A. Zeevi\thanks{Graduate School of Business, Columbia University, New York, NY 10027, USA.
e-mail: {\em assaf@gsb.columbia.edu}.}
}
\maketitle
\begin{abstract}
This paper considers a finite horizon optimal stopping problem for a sequence of independent and identically distributed random variables, where the objective is to design stopping rules that
attempt to select the random variable with the highest value in the sequence. The performance of any stopping rule may be benchmarked relative to the selection of a ``prophet"  that has perfect foreknowledge of the largest value.
Such comparisons are typically stated
in the form of ``prophet inequalities."
In this paper we develop
a game--theoretic characterization that supports a principled approach for deriving  sharp {\it non-asymptotic} prophet
inequalities for  single threshold stopping rules.
We demonstrate that sharp constants in the
ratio- and difference-type prophet inequalities
are determined by  the optimal values of
infinite two-person zero-sum game on the unit square
with particular  payoff kernels, while the the solutions to the game 
provide optimal stopping rules and  least favorable distributions.
Among other things, this formulation  also allows a systematic way to tackle  restricted classes of distributions.   The proposed framework leads to a numerically efficient algorithmic paradigm that 
allows 
computing  sharp constants in prophet inequalities with any prescribed level of accuracy. 
\end{abstract}
\vspace*{1em}
\par
\vspace*{1em}

\section{Introduction}
Optimal stopping problems have a long and storied academic history and have recently found various applications also in modern practical domains that arise in technological platforms; see further discussion and references below. At the core, the problem can be stated as a sequential selection objective: given a horizon of length $n$, one observes sequentially random variables and needs to stop them at some (random) time perceived to be associated with the largest value in the sequence.  The objective is typically to both identify such optimal stopping rules, but more practically to consider rules that are perhaps suboptimal but simple in structure, broad in scope of application, and enjoy theoretical performance guarantees.  Let us now formalize this set up.

\paragraph{Optimal stopping.} Let $X_1, \ldots, X_n$ be integrable independent  non--negative random variables
defined on the probability space $(\Omega, \sF, \bP)$
with joint distribution function $F^{(n)}:=\prod_{t=1}^nF_t$. 
Let  $\sX_t$ be the $\sigma$--field generated by $X_1, \ldots, X_t$,
$\sX_t=\sigma(X_1, \ldots, X_t)$, $1\leq t\leq n$, and let
$\sX:=\{\sX_t, 1\leq t\leq n\}$ be the corresponding filtration. By definition,
{\em a stopping time}
$\tau$ with respect to (wrt)  $\sX$ is a random
variable on $(\Omega, \sF, \bP)$ such that $\bP\{\tau \in \{1, \ldots, n\}\}=1$ and 
$\{\tau=t\} \in \sX_t$ for all $1\leq t\leq n$.
The set of all stopping times wrt filtration $\sX$ is denoted $\sT_{\rm all}$.
\par
The {\it reward} of  a given stopping time $\tau\in \sT_{\rm all}$
is defined by
\[
 V_n(\tau; F^{(n)}):= \bE X_{\tau}.
\]
The problem of optimal stopping is
to find a stopping rule $\tau^*$ such that
\begin{equation*}
 V_n(\tau^*; F^{(n)}) =
 V_n^*(\sT_{\rm all}; F^{(n)}) := \sup_{\tau \in \sT_{\rm all}} \bE X_{\tau}.
\end{equation*}
It is well known \cite{ChowRobSieg} that the optimal stopping rule $\tau^*$ is given by
\[
 \tau^* := \min\{1\leq t \leq n-1: X_t  \geq  v_{n-t}\},
\]
and $\tau^*=n$ otherwise.
 Here
the sequence of thresholds $\{v_t\}$ is defined by backward induction
\[
v_1=\bE X_n, \;\;\;v_{t+1}=\bE \{X_{n-t}\vee v_t\}, \;\;\;t=1,\ldots,n-1,
\]
and the optimal value of the problem 
is
$V_n^*(\sT_{\rm all}; F^{(n)})=v_n$.

\paragraph*{Prophet inequalities and minimax formulation.} In any specific problem instance
when $F^{(n)}$ is given,
the optimal stopping rule
$\tau^*$ and  the optimal value $V_n^*(\sT_{\rm all}; F^{(n)})$
can be computed numerically.
However,
in general,
it is difficult to assess  the performance of optimal stopping rules,
and {\em prophet inequalities} are very useful tools for this purpose.
Prophet inequalities compare  the optimal value
 of the stopping problem
 with the expected value of the maximal observation
 \[
M_n(F^{(n)}) := \bE \max_{1\leq t\leq n} X_t.
 \]
 The latter
 is the performance of the ``prophet" with complete foresight. There are two commonly used measures
 of stopping rule performance relative to the prophet: {\em the ratio--type} and {\em difference--type
 prophet inequalities}.
\par
Let  $\cF^{(n)}$ be a family of distribution functions on $[0,\infty)\times\cdots\times[0,\infty)$, and
let
$\sT$ be a class of stopping rules of $\sX$.
 The {\it competitive ratio}
of a stopping rule $\tau\in \sT$ under distribution $F^{(n)}\in \cF^{(n)}$ is
defined by
\[
 \cR_n(\tau; F^{(n)}) =  \frac{V_n(\tau;F^{(n)})}{M_n(F^{(n)})}.
\]
The competitive  ratio is well defined unless
$X_1, \ldots, X_n$ are all identically zero;
this trivial case is excluded from  consideration.
The  optimal stopping rule $\tau^*$ in class $\sT$, $\tau^*\in \sT$, for given $F^{(n)}$ satisfies
\[
 \cR_n(\tau^*; F^{(n)})=\cR_n(\sT; F^{(n)}):=\sup_{\tau\in \sT}\cR_n(\tau; F^{(n)}).
\]
A {\it ratio--type prophet inequality} associated with class of stopping rules
$\sT$ and family of distributions
$\cF^{(n)}$ is a  lower bound on
$\cR_n(\sT; F^{(n)})$ uniform over $F^{(n)}\in \cF^{(n)}$:
\begin{equation}\label{eq:gen-prophet}
 \cR_n(\sT; F^{(n)}) \geq \psi_n,\;\;\; \mbox{{\rm for all} } F^{(n)}\in \cF^{(n)},
\end{equation}
where $\{\psi_n\}$ is a numerical sequence with values in $(0, 1]$.
The {\it worst--case competitive ratio} of the optimal rule in $\sT$ over the family of distributions $\cF^{(n)}$
is
\[
\cR_n^*(\sT;\cF^{(n)}) = \inf_{F^{(n)}\in \cF^{(n)}}  \cR_n(\sT;F^{(n)})=
\inf_{F^{(n)}\in \cF^{(n)}} \sup_{\tau\in \sT}\cR_n(\tau; F^{(n)}).
\]
We say that
a ratio--type prophet inequality \eqref{eq:gen-prophet} is
 {\em asymptotically sharp} if
\begin{equation}\label{eq:gen-prophet-1}
 \lim_{n\to\infty} \big\{\psi_n^{-1} \cR_n^*(\sT; \cF^{(n)})\big\} = 1.
\end{equation}
\par
Asymptotically sharp {\it difference--type prophet inequalities}  are defined similarly. First,
consider the {\em regret} of a stopping rule $\tau\in \sT$ under $F^{(n)}\in \cF^{(n)}$
to be the difference between the prophet performance and the reward of $\tau$,
\[
 \cA_n(\tau; F^{(n)}):=M_n(F^{(n)})- V_n(\tau; F^{(n)}),
\]
and the regret of the optimal stopping rule for given $F^{(n)}$ is
\[
 \cA_n(\sT; F^{(n)}):=\inf_{\tau\in \sT}\cA_n(\tau;\cF^{(n)})=\inf_{\tau\in \sT}
 \{M_n(F^{(n)})- V_n(\tau; F^{(n)})\}.
\]
The {\it difference--type prophet inequality}
associated with class of stopping rules $\sT$ and family of probability distributions $\cF^{(n)}$ is
an upper bound
on the regret of the optimal stopping rule which holds uniformly over  $F^{(n)}\in \cF^{(n)}$ :
\begin{equation}\label{eq:gen-additive}
 \cA_n(\sT; F^{(n)}) \leq \delta_n,\;\;\;\mbox{{\rm for all }} F^{(n)}\in \cF^{(n)},
\end{equation}
where $\{\delta_n\}$ is a non--negative numerical sequence.
The worst--case regret over a family $\cF^{(n)}$ of distributions~is
\[
 \cA_n^*(\sT;\cF^{(n)}):= \sup_{F^{(n)}\in \cF^{(n)}} \cA_n(\sT; F^{(n)}),
\]
and the difference--type prophet inequality is said to be   {\it asymptotically sharp}
if
\begin{equation}\label{eq:gen-additive-1}
 \lim_{n\to\infty} \{\delta_n^{-1} \cA_n^*(\sT; \cF^{(n)})\} =1.
\end{equation}
\par
There is a great deal of interest in the derivation of asymptotically sharp prophet inequalities
and  determination of  sequences $\{\psi_n\}$ and $\{\delta_n\}$
for various families  of distributions and classes of stopping rules.
Recently there
has been a renewed interest in the topic in view of applications of optimal stopping theory
in economics and computer science. For comprehensive review  of the area  and for
additional pointers to the literature we refer to the book by
\citeasnoun{Prophet-German} and to the
surveys by \citeasnoun{Hill-Kerz},
and \citeasnoun{Correa18}.
Below we discuss selected results that are most relevant to our study.

 \paragraph*{Prophet inequalities for optimal stopping rules.}
 The classical {\em ratio--type prophet inequality} \cite{Krengel} states that
for any joint probability distribution $F^{(n)}=\prod_{t=1}^n F_t$  one has
\begin{equation}\label{eq:prophet}
 \cR_n\big(\sT_{\rm all}; F^{(n)}\big) \geq \;\frac{1}{2}.
\end{equation}
Let $\cF^{(n)}_{[0,\infty)}$ be the family of all possible
product form distributions with marginals supported on $[0,\infty)$.
Then
the prophet
inequality (\ref{eq:prophet}) is asymptotically sharp over $\cF_{[0,\infty)}^{(n)}$:
\begin{equation}\label{eq:prophet-1}
\lim_{n\to\infty}
\cR^*_n \big(\sT_{\rm all}; \cF_{[0,\infty)}^{(n)}\big)= \;\frac{1}{2}.
\end{equation}
In words, for any $F^{(n)}$ the optimal reward of stopping rules in the class $\sT_{\rm all}$ is
at least
within factor $1/2$ of the
value that can be achieved by the prophet with
complete foresight, and there exists a distribution $F_*^{(n)}$  such
that this factor cannot be improved upon asymptotically as $n$ tends to infinity.
We refer to \citeasnoun{Krengel}; see also the survey \citeasnoun{Hill-Kerz}
for  detailed discussion.
\par
As for the difference--type prophet inequalities, \citeasnoun{Hill-Kerz81}
considered the family $\cF^{(n)}_{[0,1]}$ of all product form distributions with
marginals supported on $[0,1]$ and showed that
\begin{equation}\label{eq:1/4}
 \cA_n(\sT_{\rm all}; F^{(n)}) \leq \frac{1}{4},\;\;\;\forall F^{(n)}\in \cF_{[0,1]}^{(n)}.
\end{equation}
This inequality is asymptotically sharp, i.e.,
\[
\lim_{n\to\infty} \cA^*_n\big(\sT_{\rm all}; \cF_{[0,1]}^{(n)}\big) =\frac{1}{4}.
\]
The choice of   the class $\cF_{[0,1]}^{(n)}$ was motivated by
\citeasnoun{Hill-Kerz81},
given  that if the support is unbounded then the regret can be arbitrarily large.
A case with unbounded support was in fact considered by
\citeasnoun{Kennedy-Kertz}, who focused on the  family $\cF_\sigma^{(n)}$
of all distributions $F^{(n)}=\prod_{t=1}^n F_t$
with marginals having bounded variance, ${\rm var}(X_i)\leq \sigma^2<\infty$,
$i=1, \ldots, n$. They proved that
\[
 \cA_n(\sT_{\rm all}; F^{(n)}) \leq c_n \sigma \sqrt{n-1},\;\;\;\forall
 F^{(n)}\in \cF^{(n)}_\sigma,
\]
where $c_n\leq 1/2$, $\liminf_n c_n\geq \sqrt{\ln 2- 1/2}\approx 0.439$.
However, no statement on sharpness of this inequality is made. To the best of our knowledge,
the problem of deriving such results in  more general (than bounded support) settings remains open. 
\par
If $X_1, \ldots, X_n$ is a sequence of independent identically distributed
(iid) random variables
with common distribution $F$, then the  prophet inequalities (\ref{eq:prophet}) and (\ref{eq:1/4}) can be improved.
It is worth noting however that proofs of sharpness in this case are more involved
because the family of distributions is much narrower given the homogeneity assumption.
In what follows,  in the setting of iid
random variables,
with slight abuse of notation
in all formulas we drop the superscript $(n)$ and replace  $F^{(n)}$ by $F$, $\cF^{(n)}$ by~$\cF$,~etc.
\par
In this setting, \citeasnoun{Hill-Kerz-1} show that
\[
\cR_n(\sT_{\rm all}; F)
 \geq \;\frac{1}{a_n}, \;\;\;\forall  F\in \cF_{[0,\infty)}
\]
with numerical constants $\{a_n\}$ satisfying $1.1 <a_n<1.6$.
\citeasnoun{Kerz} strengthened this result by 
providing a sharp characterization of the region where pairs
$\{M_n(F), V_n^*(\sT_{\rm all}; F)\}$  may take  values.
In particular, it follows from the results in this paper that 
\begin{equation*}
  \lim_{n\to \infty} \inf_{F\in \cF_{[0,\infty)}} \cR_n(\sT_{\rm all}; F) 
  \geq 
\frac{1}{1+\alpha_*}\approx 0.746,
\end{equation*}
where infimum  is taken over all possible distributions, and $\alpha_*$ is the unique solution to the equation
$\int_0^1 (y-y\ln y +\alpha)^{-1}\rd y=1$.
We also refer  to \citeasnoun{Prophet-German} and
\citeasnoun{CorreaPPM}.
\par
In the case of iid random variables the difference--type prophet inequality (\ref{eq:1/4}) can be
improved as well:
\citeasnoun{Hill-Kerz-1} show that for any $F$ supported on $[0,1]$
and
for some universal constants
$0<b_n<\frac{1}{4}$ with $b_2\approx 0.0625$, $b_{100}\approx 0.110$ and $b_{10,000}\approx 0.111$
one has
\[
\cA_n(\sT_{\rm all}; F) \leq b_n,\;\;\;\forall F\in \cF_{[0,1]},
\]
and this inequality is asymptotically sharp. In particular, in Proposition~5.3,
\citeasnoun{Hill-Kerz-1} present extremal
distributions for which the claimed sequence $\{b_n\}$ is attained.
\paragraph*{Prophet inequalities for single threshold stopping rules.}
One of the remarkable observations in this area is that  an inequality such as (\ref{eq:prophet}) holds even if the class
of all possible stopping rules $\sT_{\rm all}$ is replaced by a much smaller class of
simple stopping rules -- the stopping rules with a single threshold.
Specifically, let
\[ 
\sT_0:=\{\tau_0(\theta), \theta\geq 0\},\;\;\;
\sT_1:=\{\tau_1(\theta), \theta\geq 0\}
\]
 be the classes of single threshold stopping rules such that
\begin{equation}\label{eq:tau-theta}
 \tau_0(\theta) := \min\{1\leq t\leq n-1: X_t> \theta\},
 \;\;\;\;\;\;\;\;
 \tau_1(\theta) := \min\{1\leq t\leq n-1: X_t\geq \theta\},
\end{equation}
and $\tau_i(\theta)=n$, $i=0,1$  when  the
set in the parentheses is empty.
The classes $\sT_0$ and $\sT_1$ are equivalent in terms of performance of the corresponding stopping rules; taking this fact into account
in the sequel we use notation $\sT_{\rm s}$ for
$\sT_0$ or $\sT_1$.
However, sometimes it will be  convenient to distinguish between
classes $\sT_0$ and~$\sT_1$.
\par
The main observation for the single threshold stopping rules $\sT_{\rm s}$
dates back to the paper by 
\citeasnoun{Samuel-Cahn} who shows that relations (\ref{eq:prophet})--(\ref{eq:prophet-1}) remain true if
$\sT_{\rm all}$ is replaced by $\sT_{\rm s}$:
\begin{align}
 & \cR_n(\sT_{\rm s}; F^{(n)})
 \geq \;\frac{1}{2},\;\;\;\forall F^{(n)}\in \cF^{(n)}_{[0,\infty)}
 \label{eq:prophet-threshold}
 \end{align}
 and
 \begin{align*}
 & \lim_{n\to\infty}
 \cR_n^*\big(\sT_{\rm s}; \cF_{[0,\infty)}^{(n)}\big)\;
 = \;\frac{1}{2}.
\end{align*}
It is evident that  the lower bound (\ref{eq:prophet-threshold}) also holds  in
the iid case. In addition,
\citeasnoun{Samuel-Cahn} in Theorem~2 proves that this bound is asymptotically sharp
for the single threshold rules even in the iid case:
\begin{equation}\label{eq:Ester-0}
 \lim_{n\to\infty}  \cR^*_n \big(\sT_{\rm s}; \cF_{[0,\infty)}\big)
 = \;\frac{1}{2}.
\end{equation}
It turns out that
the constant $1/2$ can be improved for continuous distributions.
In particular, let
\[
 \cC_{[0,\infty)}:= \big\{F: \;\;F \hbox{  is continuous on  } [0,\infty)\big\};
\]
then
it follows from results in \citeasnoun{CorreaPPM}
that
\[
 \lim_{n\to \infty}\cR_n^*\big(\sT_{\rm s}; \cC_{[0,\infty)} \big) \geq 1- \frac{1}{e}.
\]
We refer to
\citeasnoun{Ehsani} for the analysis of the prophet inequality in the iid setting with
constant $1-e^{-1}$, but for randomized single threshold
stopping rules. 
\par
We are not aware of any analogous results for difference--type prophet inequalities that establish   optimality of 
{\it single threshold}  stopping rules.  A more detailed discussion of sharpness 
of prophet inequalities in the iid setting for single threshold 
stopping rules is given below in Section~\ref{sec:sharpness}.

\paragraph*{Main contributions.}

The literature on prophet inequalities is 
 vast and the topic has been extensively studied. 
Sharp prophet inequalities presented in the literature, are typically 
{\em asymptotically sharp} in the sense of definitions
\eqref{eq:gen-prophet-1} and \eqref{eq:gen-additive-1}, and they are derived 
for classes of all possible distributions $\cF_{[0,\infty)}$
or $\cF_{[0,1]}$. The process usually follows two main steps:  
a stopping rule in a given class is proposed
and  a lower (upper) bound on the competitive ratio (regret) of the rule is derived for any
distribution from the given class; then, a least favorable distribution is exhibited for which 
the derived bounds are (asymptotically) achieved.
In general, very little is known about
exact values of the worst--case competitive ratio and the worst--case regret
in the non--asymptotic setting when
the horizon~$n$
is fixed, or when one considers additional constraints on the
class  of underlying distributions. 
\par
In this work we focus on the iid setting and
develop a unified and principled framework for
the derivation of non--asymptotic sharp prophet inequalities for
single threshold stopping rules.
We study randomized single threshold stopping rules that aggregate stopping rules 
$\tau_0(\theta)$ and
$\tau_1(\theta)$ in \eqref{eq:tau-theta}
and
allow one to
achieve much better performance relative to the prophet.
We  show that
for such rules  the derivation of sharp 
prophet inequalities is equivalent to solving a two--person zero--sum infinite game on the unit square.
Then for any fixed problem horizon~$n$
the optimal value of the aforementioned game yields
 the sharp constant in the non--asymptotic prophet inequality, while
the solution of the game provides 
both the least favorable distribution as well as  the corresponding 
optimal single threshold
stopping rule.
The developed framework supports
simple computational procedures
to derive
sharp non--asymptotic prophet
inequalities for restricted classes of distributions,
characterizes numerical (discretization) errors in said computation, and illustrates
efficacy on an array of test problems.

\paragraph*{Notation.}
Throughout the paper we use the following notation.
Let $F$ be a distribution function; then
we define the quantile function
$F^{\leftarrow}: [0,1]\to \bar{\bR}=[0, \infty]$
of $F$ by
\[
  F^\leftarrow (t):=\inf\{x\in [0,\infty): F(x)\geq t\},\;\;
  0 \leq t \leq 1.
 \]
The quantile function $F^\leftarrow$
is
 left continuous inverse to $F$.
 We also denote $U(t)$, $t\geq 1$, the $(1-1/t)$--quantile of $F$:
 \begin{equation}\label{eq:U}
 U(t):= \bigg(\frac{1}{1-F}\bigg)^{\leftarrow}(t) = F^\leftarrow(1-\tfrac{1}{t})=\inf\{x\in [0,\infty):
 F(x)\geq 1-\tfrac{1}{t}\},\;\;\;t\geq 1.
\end{equation}
In what follows
the range of a distribution function is denoted
${\rm range}(F):=\{F(x): x\in [0, \infty)\}$.
A  {\em realizable $t$--quantile of $F$} is any value  $z\in [0,\infty)$ such that $F(z)=t$.
\paragraph*{Organization.} The rest of this paper is structured as follows.  In Section~\ref{s-single} we provide the  results concerning performance of single threshold rules in the iid setting  and related prophet inequalities.
Section~\ref{sec:sharpness} discusses various 
statements  about sharpness of prophet inequalities for single
threshold rules that have been
made in antecedent literature.
The main results of this paper are presented in Sections~\ref{s-game}--\ref{s-num}.
Section~\ref{s-game} develops  a game--theoretic formulation that characterizes 
sharp prophet inequalities and corresponding  least favorable distributions and
optimal single threshold stopping rules. In Section \ref{s-prophet} we detail
how this game--theoretic characterization can be leveraged towards
theoretical guarantees, and discuss
the computational aspects related to the underlying optimization problems.
Section~\ref{s-num} contains discussion and examples that illustrate application of the developed
approach for restricted families of probability distributions. Section~\ref{sec:proofs} presents proofs of the
main results of this paper.

\section{Single threshold stopping rules}\label{s-single}

We slightly extend the definition of 
the single threshold stopping rules in \eqref{eq:tau-theta}
to allow randomization. Let $\{\xi_t, 1\leq t\leq n\}$ be
iid Bernoulli random variables with success probability $p\in [0,1]$, independent of $\sX$.
Define
\begin{align}\label{eq:stop-random}
 \tau_p(\theta):= \min\big\{1\leq t\leq n-1: (X_t>\theta)\cup (X_t=\theta,\xi_t=1)\big\},
\end{align}
and $\tau_p(\theta)=n$ if the set in the parentheses is empty. In words,
$\tau_p(\theta)$ stops at the first time~$t$ when the observed value $X_t$ exceeds the
fixed threshold $\theta$ or if   $X_t=\theta$ and the outcome $\xi_t$ of the independent
Bernoulli trial with success probability $p$ equals one. Let
\[ 
\sT_{\rm s, r}:=\{\tau_p(\theta): \theta\geq 0,\;p\in [0,1]\} 
\]
be the set of all such stopping rules.
The stopping rule $\tau_p(\theta)$ is determined by
two parameters: the threshold $\theta$ and the success probability~$p$.
Note that $\sT_0$ and $\sT_1$ are subsets of $\sT_{\rm s,r}$ corresponding to
$p=0$ and $p=1$ respectively.
\par
The randomization in the construction of $\tau_p(\theta)$ admits
very simple interpretation in terms of the stopping rules $\tau_0(\theta)$
and $\tau_1(\theta)$ defined in \eqref{eq:tau-theta}.
If, for a given threshold~$\theta$, both rules 
$\tau_0(\theta)$ and $\tau_1(\theta)$ prescribe to stop
at the same time instance $t=1, \ldots, n-1$ then 
$\tau_p(\theta)=t$ as well. But if $\tau_1(\theta)=t$ and 
$\tau_0(\theta)>\tau_1(\theta)$ then the decision to stop or not to stop at $t$ is made according to the outcome of the independent Bernoulli experiment with success probability~$p$. 


We begin with the result that establishes
the exact formula for the reward of the 
optimal single threshold
stopping rule
from $\sT_{\rm s, r}$.

\begin{theorem}\label{th:prophet}
{\rm \textbf{(Performance of optimal single threshold rules)}}
Let $X_1, \ldots, X_n$ be iid random variables with common distribution $F$, and
\begin{equation}\label{eq:Fp}
\Delta(x):= F(x)-F(x-),\;\; F_p(x):= pF(x-)+ (1-p)F(x),\;\;\;\forall x,\;\;p\in [0,1].
\end{equation}
Then
 \begin{align}
& V_n^*(\sT_{\rm s, r};F)
\nonumber
\\
  &\;= \sup_{\theta\geq 0,\, p\in[0,1]}\;\bigg\{
   [1-F_p^{n}(\theta)]
\Big[ \theta + \frac{\int_{\theta}^\infty (1-F_p(x))\rd x}{1-F_p(\theta)}\Big]
+ F_p^{n-1}(\theta)\Big[\int_{[0, \theta]}  x\, \rd F(x) - p \theta  \Delta(\theta)\Big]
\bigg\}
\label{eq:iid}
\\*[2mm]
&\; = \sup_{\theta\geq 0,\, p\in[0,1]}\;\bigg\{
  [1-F_p^{n-1}(\theta)]
\Big[ \theta + \frac{\int_{\theta}^\infty (1-F_p(x))\rd x}{1-F_p(\theta)}\Big]
+ F_p^{n-1}(\theta)\int_0^\infty [1-F_p(x)]\rd x
\bigg\}.
\label{eq:iid-1}
 \end{align}
\end{theorem}
\begin{remark}
 Letting $p=0$ or $p=1$ in \eqref{eq:iid}--\eqref{eq:iid-1}
 of Theorem~\ref{th:prophet} we obtain the exact formulas for the reward of 
 the optimal single threshold stopping rules from $\sT_{\rm s}$.
\end{remark}
\begin{remark}
Since $M_n(F) \geq V_n^*(\sT_{\rm all};F)\geq V_n^*(\sT_{\rm s, r};F)$, the theorem  implies lower bounds
on $M_n(F)=\bE \max_{1\leq t\leq n} X_t$, and
on the performance of the optimal stopping rule
in the class of all stopping rules,
$V_n^*(\sT_{\rm all};F)$.
It is worth noting that (\ref{eq:iid})  provides
an  improvement of Markov's inequality for
 random variables of the type $\max_{1\leq t\leq n}  X_t$,
 where $X_1, \ldots, X_n$ are non--negative  iid random variables.
 To see this, observe that  for $p=0$ we have $F_0=F$,  
 $1-F^n(\theta)=\bP\{\max_{1\leq t\leq n} X_t>\theta\}$, $\forall \theta$, and
 $M_n(F)\geq V_n^*(\sT_{\rm s,r}; F)$. 
 \end{remark}
 \par
 Some implications of
 Theorem~\ref{th:prophet}
 are contained  in the following two corollaries.

 \begin{corollary}\label{cor:1}
 Let $X_1, \ldots, X_n$ be iid random variables with common distribution
 $F\in \cF_{[0,\infty)}$.
 Then
 \begin{equation}\label{eq:cor-1}
  V_n^*(\sT_{\rm s, r}; F)\geq \big[1-(1-\tfrac{1}{n})^{n-1}\big] M_n(F) + \big(1-\tfrac{1}{n})^{n-1}M_1(F).
\end{equation}
Moreover, for every $n$ one has
\begin{equation}\label{eq:cor-11}
 \cR_n(\sT_{\rm s, r};F)\geq  1- \Big(1-\frac{1}{n}\Big)^n\geq  1- \frac{1}{e},
 \;\;\;\forall F\in \cF_{[0,\infty)}.
\end{equation}
\end{corollary}
\begin{proof}
Let  $\theta=U(n)$, and
\begin{equation}\label{eq:p}
 p=\left\{\begin{array}{cl}
0, & F(U(n))= 1-1/n,\\[2mm]
\frac{F(U(n))-(1-\frac{1}{n})}{F(U(n))- F(U(n)-)}, & {\rm otherwise},
          \end{array}
\right.
\end{equation}
where $U(t)$ is the $(1-1/t)$--quantile
of distribution $F$ [see  \eqref{eq:U}].
With this choice of $p$, $F_p(U(n))=1-1/n$, and therefore
it follows from
(\ref{eq:iid-1}) that
\begin{align}
 &V_n^*(\sT_{\rm s, r}; F) \geq V_n(\tau_p(U(n)); F)
 \nonumber
 \\
 &= \big[1-F_p^{n-1}(U(n))\big]
 \Big\{U(n)+ n \int_{U(n)}^\infty [1-F(x)]\rd x\Big\}
 + F_p^{n-1}(U(n))\int_0^\infty [1-F(x)]\rd x.
 \label{eq:111}
 \end{align}
Note also that
 \begin{equation}\label{eq:Mn}
  M_n(F)=\int_0^\infty [1-F^n(x)]\rd x\leq U(n)+n\int_{U(n)}^\infty [1-F(x)]\rd x.
 \end{equation}
Therefore
 \begin{align*}
 &V_n^*(\sT_{\rm s, r}; F) \geq V_n(\tau_p(U(n)); F)
=
 \big[1-(1-\tfrac{1}{n})^{n-1}\big] M_n(F)
 + (1-\tfrac{1}{n})^{n-1} M_1(F),
 \end{align*}
as claimed in \eqref{eq:cor-1}.
The same chain of inequalities applied to \eqref{eq:iid}
leads to
\begin{equation}\label{eq:VM}
 V_n^*(\sT_{\rm s, r}; F) \geq
 \big[1-(1-\tfrac{1}{n})^{n}\big] M_n(F)
 + (1-\tfrac{1}{n})^{n-1} \int_{[0,U(n))} x\rd F(x),
\end{equation}
which implies  \eqref{eq:cor-11}.
\end{proof}
The proof shows that inequality \eqref{eq:cor-11}
holds for the stopping rule $\tau_p(\theta)$ with
threshold $\theta=U(n)$
and parameter $p$ defined in \eqref{eq:p}.
Note that if the class
$\sT_{\rm s}$ is considered, then \eqref{eq:cor-11} is still
fulfilled, but only for all distributions $F$  with the realizable $(1-1/n)$--quantile. Thus,
the following non--asymptotic prophet inequality holds
\begin{equation*}
 \cR^*_n\big(\sT_{\rm s}; \cG^n\big) \geq
 1- \Big(1-\frac{1}{n}\Big)^n \geq
 1-\frac{1}{e},\;\;\;\forall n,
\end{equation*}
where
\begin{equation}\label{eq:G-n}
 \cG^n:=\Big\{F\in \cF_{[0,\infty)}: 1-\frac{1}{n}\in {\rm range}(F)\Big\}.
\end{equation}
Note that class $\cG^n$ is rather rich: it includes
all continuous distributions on $[0, \infty)$, and discrete distributions with  realizable $(1-1/n)$--quantile.
In particular,
if $F\in \cC_{[0,\infty)}$
then, obviously,
${\rm range}(F)\supseteq [0,1)$
and  $\cC_{[0,\infty)} \subset  \cG^n$ for all $n$.
Therefore
\[
 \cR_n^*\big(\sT_{\rm s}; \cC_{[0,\infty)}\big)
 \geq
 1- \Big(1-\frac{1}{n}\Big)^n \geq 1-\frac{1}{e},\;\;\;\forall n.
\]
The above prophet inequality holds {\em asymptotically} as $n\to\infty$
under much weaker conditions when $F$ is continuous
on the right tail.
In particular, if for $x_0>0$
\[
 \cC_{[0,\infty)}(x_0):=\big\{ F\in \cF_{[0,\infty)}:
 F
 \hbox{ is continuous on } [x_0, \infty)\big\},
\]
then for any $x_0>0$
there exists $n_0=n_0(x_0)$ such that $1-1/n \in {\rm range}(F)$ for all
$n\geq n_0$.  Therefore
\[
 \liminf_{n\to\infty} \cR_n^*\big(\sT_{\rm s}; \cC_{[0,\infty)}(x_0)\big)
 \geq 1-\frac{1}{e}.
\]
\begin{remark}
 The inequality \eqref{eq:VM} implies 
 that 
 $
  M_n(F)-V_n^*(\sT_{\rm s, r}; F) \leq \big(1- \tfrac{1}{n}\big)^n M_n(F)$, and 
   therefore 
  \begin{equation}\label{eq:A-inequality}
   \cA_n^*(\sT_{\rm s, r}; \cF_{[0,1]}) \leq \Big(1- \frac{1}{n}\Big)^n.
  \end{equation}
A similar difference--type prophet inequality 
with constant $(1-1/n)^n$ 
has been established in the 
literature 
for 
the problem of
stopping the sequence $\{Y_t=X_t- ct: t=1, \ldots, n\}$ 
with the cost of observation $c\geq 0$, where  
$X_1, \ldots, X_n$ are independent \cite{Jones} 
or iid \cite{Ester2} random variables. It is worth noting 
that the optimal rule in the stopping problem with the 
iid random variables and 
cost of observations
is the single threshold stopping rule; so 
there is a close connection between  
\eqref{eq:A-inequality} and the results in \citeasnoun{Ester2}.
It is proved in \citeasnoun{Ester2} that the constant $(1- 1/n)^n$  
is sharp. However, sharpness of the inequality  is understood  
not only in the sense of  the least favorable distribution  $F$, 
but also in the sense of the least favorable observation cost~$c$. In fact, 
our results below demonstrate that inequality 
\eqref{eq:A-inequality} can be improved 
when there is no observation 
cost, $c=0$.
\end{remark}
\par
 A byproduct of the proof of Corollary~\ref{cor:1} is
the following
distribution--free
inequality
on the sequence $\{M_n(F), n\geq 1\}$ of maximal order statistics corresponding to
sample sizes $n=1,2,\ldots$.
This result  is interesting in its own right.
\begin{corollary}\label{cor:2}{\rm \textbf{(Distribution
free bounds on the growth of expected extremes)}}
Let $X_1, X_2, \ldots$ be iid random variables with
  common distribution function $F$, and let
  $M_t(F):=\bE\max_{1\leq i\leq t} X_i$ for any $t \geq 1$.
  Then  for
integers $n\geq 1$ and $k\geq 0$ and any $F \in \cF_{[0,\infty)}$
one has
\begin{equation*}
 M_n(F) \geq  (1-\lambda_{n,k})  M_{n+k}(F) +
 \lambda_{n,k}  M_1(F),\;\;\;
 \;\;\;\lambda_{n,k}:= (1-\tfrac{1}{n+k})^{n-1}.
\end{equation*}
\end{corollary}
\begin{proof}
It follows from \eqref{eq:iid-1} that for any $\alpha\geq 1$ and $p\in [0,1]$
one has
\begin{align*}
 M_n(F)  &\geq V_n(\tau_p(U(\alpha)); F)
 \\
& = \big[1-F_p^{n-1}(U(\alpha))\big]
 \Big\{U(\alpha)+  \frac{\int_{U(\alpha)}^\infty [1-F(x)]\rd x}{1-F_p(U(\alpha))}\Big\}
 + F_p^{n-1}(U(\alpha))\int_0^\infty [1-F(x)]\rd x.
 \end{align*}
Choosing $\alpha=n+k$ and  $p$ such that $F_p(U(n+k))= 1- 1/(n+k)$ [see \eqref{eq:p}]
in the previous formula 
and applying \eqref{eq:Mn}
with $n$ replaced by $n+k$
we complete the proof.
\end{proof}

\begin{remark}
\citeasnoun{Downey} study the rate of growth of the sequence $\{M_n(F), n\geq 1\}$,
 and define ordering of distribution functions based on $\{M_n(F), n\geq 1\}$. Among other results,
 this paper proves the inequality
 \[
  M_n(F) \geq (1-e^{-1}) \Big\{U(n) + n\int_{U(n)}^\infty [1-F(x)]\rd x \Big\},\;\;
  \forall F\in \cF_{[0,\infty)}.
 \]
 This inequality 
is an immediate  consequence of \eqref{eq:iid} 
with $\theta=U(n)$ and $p$ given in \eqref{eq:p}.
For additional results
on behavior of the sequence of expectations of the maximum of  iid random variables
we refer to \citeasnoun{Downey-1}.
\end{remark}

 \section{Discussion}
 \label{sec:sharpness}
Some statements about sharpness of prophet inequalities for  single threshold stopping rules appeared in  
previous literature.
Our discussion of these statements is divided in two parts. In the first part
 of this section we discuss
 the results of  \protect\citeasnoun{Samuel-Cahn},
 while the second part deals with results from the online auctions literature.


\subsection{Results of \protect\citeasnoun{Samuel-Cahn}}
It is shown in Theorem~2 in \citeasnoun{Samuel-Cahn} that
for the class $\sT_{\rm s}$ of single threshold rules without
randomization
one has 
 \begin{equation}\label{eq:Ester}
  \lim_{n\to\infty} \cR^*_n(\sT_{\rm s}; \cF_{[0,\infty)}) = \frac{1}{2}.
 \end{equation}
In the proof  of \eqref{eq:Ester}
\citeasnoun{Samuel-Cahn}  considers a discrete
least favorable
distribution $F_*$
 having three atoms at $0$, $a\in (0,1)$ and $1$ with probabilities
$1-(b+c)/n$, $c/n$ and $b/n$, respectively; here $b>0$, $c>0$ and $b+c<n$.
Thus,  the distribution $F_*$ is
\[
 F_*(x) = \left\{ \begin{array}{ll}
                  1- \frac{b+c}{n}, & 0\leq x <a,\\[2mm]
                  1-\frac{b}{n}, & a\leq x< 1,\\[2mm]
                  1, & x\geq 1.
                  \end{array}
\right.
\]
In this example a straightforward calculation yields
\begin{align*}
M_n(F_*) = a\bigg[\Big(1-\frac{b}{n}\Big)^n-\Big(1-\frac{b+c}{n}\Big)^n\bigg]
+1-\bigg(1-\frac{b}{n}\bigg)^n.
\end{align*}
For the stopping rule $\tau_0(\theta)$ in \eqref{eq:tau-theta} there are three
possible thresholds $\theta=0$, $\theta=a$ and $\theta=1$.
In view of~(\ref{eq:iid-1}) with $p=0$,
the corresponding
values are given by
\begin{align*}
&
\bE X_{\tau_0(0)} = \bigg[1-\Big(1-\frac{b+c}{n}\Big)^{n-1}\bigg]\frac{ac+b}{c+b}
+\Big(1-\frac{b+c}{n}\Big)^{n-1}\frac{ac+b}{n},
\\[2mm]
&
\bE X_{\tau_0(a)} =1-\Big(1-\frac{b}{n}\Big)^{n-1}
+\Big(1-\frac{b}{n}\Big)^{n-1}\frac{ac+b}{n},
\\
&
\bE X_{\tau_0(1)}=\bE(X_1)=\frac{ac+b}{n}.
\end{align*}
Now, put
$a= a_n= 1/n$, $b=b_n= 1/n$, and $c=c_n=\sqrt{n}$,
and let $F_n$ stand for the distribution $F_*$ with parameters $a_n$, $b_n$ and $c_n$.
Then, asymptotically,
\[
 M_n(F_n)\;\sim\; \frac{2}{n},\;\;\;\bE X_{\tau_0(0)}\;\sim\; \frac{1}{n},
 \;\;\;
 \bE X_{\tau_0(a_n)}\; \sim\; \frac{1}{n}, \;\;\;\bE X_{\tau_0(1)}\sim \frac{1}{n^{3/2}},
\]
where $v_n\sim w_n$ means that $\lim_{n\to\infty}(v_n/w_n)=1$.
Thus,
the limit $1/2$ is achieved asymptotically for
$\cR_n(\sT_{\rm s}; F_n)$.
\par 
It is instructive to calculate the competitive ratio of the
best stopping rule  from $\sT_{\rm s,r}$  (the best rule with randomization)
on the sequence of  
distributions $\{F_n\}$ defined above.
Consider the stopping rule $\tau_p(\theta)$ defined in \eqref{eq:stop-random}
with 
\[
 \theta=a_n=\frac{1}{n},\;\;p=p_n=
 \frac{F_n(a_n)-(1-\frac{1}{n})}{F_n(a_n)-F_n(a_n-)}=\frac{1-b_n}
 {c_n}=\frac{1-\frac{1}{n}}{\sqrt{n}}.
\]
With this choice of parameters, in view of   \eqref{eq:iid}
\begin{align*}
 \bE X_{\tau_{p_n}(a_n)} &=
 \Big[1-\Big(1-\frac{1}{n}\Big)^n\Big]\Big[a_n+ n\int_{a_n}^\infty (1-F_n(x))\rd x\Big]
 + \Big(1-\frac{1}{n}\Big)^{n-1}\Big[\int_{[0,a_n]} x \rd F_n(x) - p_na_n \Delta(a_n)\Big]
 \\
 &= \Big[1-\Big(1-\frac{1}{n}\Big)^n\Big]\big[a_n + b_n(1-a_n)\big]
 +\Big(1-\frac{1}{n}\Big)^{n-1} \Big[\frac{a_n c_n}{n} - \frac{a_n(1-b_n)}{n}\Big]
 \\
 &=
 \Big[1-\Big(1-\frac{1}{n}\Big)^n\Big]\Big(\frac{2}{n}- \frac{1}{n^2}\Big)
 + \Big(1-\frac{1}{n}\Big)^{n-1} \frac{1}{n\sqrt{n}} 
 \Big[1-\frac{1}{\sqrt{n}}+\frac{1}{n\sqrt{n}}\Big].
\end{align*}
Since $M_n(F_n)\sim 2/n$, on the sequence of distributions $\{F_n\}$ we obtain
\[
 \lim_{n\to\infty} \frac{\bE X_{\tau_{p_n}(a_n)}}{M_n(F_n)} = 1-\frac{1}{e}.
\]
Thus,
\[
 \lim_{n\to\infty} \cR_n^*(\sT_{\rm r, s}; \cF_{[0,\infty)}) = 1-\frac{1}{e},
\]
i.e., the prophet inequality \eqref{eq:cor-11} is asymptotically sharp
for the single threshold stopping rules with randomization, and it is achieved on the same sequence of worst--case distributions as the prophet inequality~\eqref{eq:Ester}.

\subsection{Online auctions literature}  
Recently there has been a renewed interest in  prophet inequalities due to
relations between optimal stopping theory and design of  online auction mechanisms.
In one  such scenario, a seller has a single item to sell to~$n$ customers where  the
$i$-th customer, $i=1, \ldots, n$, has  random valuation $X_i$ for the item. The valuations of different customers are assumed to be independent. The customers arrive sequentially, and they are
presented with a price
(perhaps, customer-dependent).  The customer purchases the item if
his/her valuation exceeds the quoted price.
The goal of the seller is to
design  a  pricing policy
so as to maximize
the revenue. This setting is often referred to as a {\em posted-price auction}.
\par
{\em The prophet secretary} model in
 \citeasnoun{Esfandiari} and \citeasnoun{Ehsani}
assumes that valuations
$X_1, \ldots, X_n$ are independent random variables with distributions
$F_1, \ldots, F_n$, and customers arrive sequentially in a random order.
This implies  that the realized sequence of valuations is $X_{\pi_1}, \ldots, X_{\pi_n}$,
where $\pi=(\pi_1, \ldots, \pi_n)$ is a random permutation of
$\{1, \ldots, n\}$, independent of $X_1, \ldots, X_n$.
Various pricing policies were studied in this setting
and corresponding
prophet inequalities were derived.
We refer to \citeasnoun{Correa18} for review of results in this area.
\par
 \citeasnoun{CorreaPPM} showed
that for continuous distribution functions one has 
\begin{equation}\label{eq:1-1/e}
 \cR_n(\sT_{\rm s}; F) \geq 1-\frac{1}{e},\;\;\;\forall F\in
 \cC_{[0,\infty)}.
\end{equation}
A similar result has been obtained in
\citeasnoun{Ehsani} for the class of all distributions $\cF_{[0,\infty)}$,
but that paper considered stopping rules with random breaking of ties
when the observed random variable
$X_t$ is {\it exactly equal to} a chosen threshold.
\citeasnoun{Ehsani} (see their Theorem~21)  claim that the prophet inequality
$\cR_n^*(\sT_{\rm s}; \cF_{[0,\infty)})\geq 1-e^{-1}$
is asymptotically sharp in the iid setting. In the proof of this statement
the authors first consider
the distribution $F_*$ having two atoms at $n/(e-1)$ and $(e-2)/(e-1)$ with
probabilities $1/n^2$ and $1-1/n^2$ respectively.
For this distribution there are two possible choices of the threshold:
$\theta_1=0$
and  $\theta_2= (e-2)/(e-1)$. Then \citeasnoun{Ehsani} show that the best
single threshold stopping rule achieves the competitive ratio at most $0.58$. Furthermore, the authors argue that this bound
can be improved by randomization to \mbox{$1-e^{-1}$}.
\par
For the distribution $F_*$ we have
\begin{align*}
 M_n(F_*) = \frac{e-2}{e-1}\Big(1-\frac{1}{n^2}\Big)^n +
 \frac{n}{e-1}\Big[1- \Big(1-\frac{1}{n^2}\Big)^n\Big].
\end{align*}
Using (\ref{eq:iid-1}) for the single threshold rule with no randomization ($p=0$),
we obtain
\begin{align*}
 \bE X_{\tau_0(\theta_1)} = \frac{e-2}{e-1}\Big(1-\frac{1}{n^2}\Big) + \frac{1}{n(e-1)},
\end{align*}
and
\begin{align*}
 \bE X_{\tau_0(\theta_2)} =
 \frac{n}{e-1}\bigg[1-\Big(1-\frac{1}{n^2}\Big)^{n-1}\bigg] +
 \bigg[ \frac{e-2}{e-1}\Big(1-\frac{1}{n^2}\Big) + \frac{1}{n(e-1)}\bigg]
 \Big(1-\frac{1}{n^2}\Big)^{n-1}.
\end{align*}
These formulas imply that
\[
 \lim_{n\to\infty} M_n(F_*)=1,\;\;\lim_{n\to\infty} \bE X_{\tau_0(\theta_1)} =
 \frac{e-2}{e-1},\;\;
 \lim_{n\to\infty} \bE X_{\tau_0(\theta_2)} = 1,
\]
so that for the specified distribution
$\lim_{n\to\infty} \cR_n(\sT_{\rm s}; F_*)=1$. Thus,  for the distribution
$F_*$  given in \protect\citeasnoun{Ehsani} the 
single threshold stopping rules with no randomization 
(from $\sT_{\rm s}$)
achieve a  factor of one as $n\to\infty$.
In fact, for any distribution with two atoms
the single threshold rules achieve the factor one;
this has been already pointed out in Remark~1 in \citeasnoun{Samuel-Cahn}.
\par
The difference between the above calculation  and the conclusion
in \citeasnoun{Ehsani} stems from the fact that,
in spite of the
close connection between optimal stopping and online auctions, the
algorithms considered in the latter area [see, e.g.,
\citeasnoun{Esfandiari},
\citeasnoun{Ehsani}, \citeasnoun{CorreaPPM}, etc.]
are not completely equivalent to those in  optimal stopping problems
[see, e.g., \citeasnoun{ChowRobSieg}, \citeasnoun{hill-kertz81b}, \citeasnoun{Samuel-Cahn},
\citeasnoun{Ber-Gnedin},
\citeasnoun{Hill-Kerz}, \citeasnoun{BF1996}, \citeasnoun{AGS2002}, etc.].
In particular, the decision rules in
\citeasnoun{Esfandiari},
\citeasnoun{Ehsani} and  \citeasnoun{CorreaPPM} (which are fully in line with various other papers in this strand of literature)
are left undefined on the event where the random variables $X_1, \ldots, X_n$
do not exceed
the corresponding thresholds [see, e.g., Algorithm Prophet Secretary on page~1687 in \citeasnoun{Esfandiari}
and Algorithm~1 in \citeasnoun{CorreaPPM}]. Strictly speaking, the considered decision rules are not stopping times because 
they do not take values in the stopping set $\{1, \ldots, n\}$.
Effectively,  the implication of this fact is
that if $X_1, \ldots, X_n$ are below the respective thresholds then the
obtained reward
is equal to zero.
This formally corresponds to computation of
the reward of a single threshold rule without
the last term on the right hand side
of (\ref{eq:iid}). This  distinction 
affects statements about
sharpness of prophet inequalities derived in the  online auctions literature, and the comparison  across literatures should take this into consideration.
In contrast,
 stopping times are naturally defined to be equal to the terminal 
 value $X_n$ of the sequence at the end of the horizon~$n$
 on the event when all observations are less than the
 corresponding thresholds, i.e.,
the last observation is selected. 
All this implies that 
sharpness of the inequality 
\eqref{eq:cor-11} for single threshold stopping times does not follow from the results in \citeasnoun{Ehsani}.
\par 
In the next section we demonstrate that
derivation of sharp non--asymptotic prophet inequalities
for the stopping rules from the class  $\sT_{\rm s,r}$
is equivalent to the solution of an infinite  two--person  zero--sum 
game on the unit square with
particular  payoff kernels.
The value of this game provides a sharp constant in the prophet inequality, while
the optimal solution yields the corresponding 
least favorable distribution and 
optimal single threshold stopping rule.

\section{Game--theoretic characterization of prophet inequalities}
\label{s-game}
Let us  introduce
the following notation. 
For $(x, y)\in [0,1]\times [0,1]$ define
 \begin{align}\label{eq:R}
 &R(x, y) :=\frac{1-x^{n-1}}{1-y^n} \min \Big\{1, \frac{1-y}{1-x}\Big\} +x^{n-1}
  \frac{1-y}{1-y^n}=
  \left\{\begin{array}{ll}
\frac{1-x^{n-1}y}{1-y^n}, & x>y,\\*[2mm]
\frac{1-y}{1-y^n}\frac{1-x^n}{1-x}, & x < y,
\\*[2mm]
1, & x=y,
         \end{array}
\right.
\end{align}
and
\begin{align}
 A(x,y):= 1-y^n - (1-x^{n-1}) \min \Big\{1, \frac{1-y}{1-x}\Big\}
 & - x^{n-1}(1-y)
\nonumber
 \\[3mm]
 &=\left\{ \begin{array}{ll}
           y(x^{n-1}-y^{n-1}), & x>y\\[2mm]
           (1-y) \sum_{k=1}^{n-1} (y^k-x^k), & x\leq y.
          \end{array}
\right.
\label{eq:A}
\end{align}
For probability distributions $\lambda$ and $\mu$ on $[0,1]$ we put
\begin{equation}\label{eq:bar-R-A}
 \bar{R}(\lambda, \mu) :=\int_{0}^1\int_0^1 R(x, y) \rd\lambda(x) \rd \mu(y),\;\;
 \;\;\bar{A}(\lambda, \mu) :=\int_0^1 \int_0^1 A(x, y) \rd\lambda(x) \rd \mu(y).
\end{equation}
The main result of this section is given in the next theorem.
\begin{theorem}\label{th:R-A}{\rm \textbf{(Saddlepoint characterization)}}
For any fixed $n$  the following representations hold:
 \begin{align}
  \label{eq:R*}
 & \cR^*_n\big(\sT_{\rm s, r}; \cF_{[0,\infty)}\big) =
  \inf_\mu \sup_\lambda \Big\{\bar{R}(\lambda,\mu): \lambda\in \cF_{[0,1]},\;\mu
  \in \cF_{[0,1]}\Big\},
  \\
& \cA^*_n\big(\sT_{\rm s, r}; \cF_{[0,1]}\big) =
\sup_\mu \inf_\lambda \Big\{\bar{A}(\lambda, \mu):
\lambda\in \cF_{[0,1]},\; \mu\in \cF_{[0,1]}\Big\},
\label{eq:A*}
  \end{align}
where $\cF_{[0,\infty)}$ and $\cF_{[0,1]}$
are the classes of all distributions on the respective domain.
\end{theorem}
Theorem~\ref{th:R-A} provides characterization of
 the worst--case competitive ratio
 $\cR^*_n\big(\sT_{\rm s,r}; \cF_{[0,\infty)}\big)$
and the worst--case regret
$\cA^*_n\big(\sT_{\rm s, r}; \cF_{[0,1]}\big)$
  via  two-person zero-sum infinite games on
the unit square.
We refer to \citeasnoun{Karlin}
for detailed discussion of such problems.
 The interpretation of the games
 in  \eqref{eq:R*}--\eqref{eq:A*} is evident: the
 decision maker  generates a random value  $x$ from distribution
 $\lambda$ on $[0,1]$ and sets
 the stopping rule threshold $\theta_x= F^\leftarrow(x)$
 and the success probability $p_x= (F(\theta_x)-x)/(F(\theta_x)-F(\theta_x-))$, while the nature
 selects the distribution $\mu$ on $[0,1]$ which
 is directly related to the quantile function of $F$.
In particular, as the proof of Theorem~\ref{th:R-A} shows,
for the ratio--type prophet inequality the quantile function
$F^{\leftarrow}$ of the least favorable distribution is related to the optimal solution  of the game through the  relationship $\rd \mu (y):=(1-y^n) \rd F^{\leftarrow}(y)$,
$\forall y\in [0,1]$, while for the difference--type prophet inequality we have
$\rd \mu(y)= \rd F^{\leftarrow}(y)$, $y\in [0,1]$.
 \par
Theorem~\ref{th:R-A} enables us to establish minimax results
for the ratio--type and difference--type prophet inequalities.
\begin{corollary}\label{cor:minimax}{\rm \textbf{(Interchange results and minimax)}}
The infinite games (\ref{eq:R*}) and  (\ref{eq:A*}) have solutions,~i.e.,
\begin{align}\label{eq:R-minimax}
  &\inf_{F\in \cF_{[0,\infty)}}\;
  \sup_{\tau \in \sT_{\rm s,r}} \cR_n(\tau;F) =
 \sup_{\tau \in \sT_{\rm s, r}}\inf_{F\in  \cF_{[0,\infty)}} \cR_n(\tau;F),
\\
\label{eq:A-minimax}
 &\sup_{F\in \cF_{[0,1]}}\;
  \inf_{\tau \in \sT_{\rm s, r}}\cA_n(\tau;F) =
 \inf_{\tau \in \sT_{\rm s, r}}\sup_{F\in \cF_{[0,1]}} \cA_n(\tau;F).
 \end{align}
\end{corollary}
The existence of the value of the infinite game (\ref{eq:A*})
follows from continuity of the payoff kernel $A(x,y)$ on $[0,1]\times [0,1]$;
see, e.g., \citeasnoun[Section~4.5]{Kuhn}. Hence \eqref{eq:A-minimax} holds.
The kernel
$R(x,y)$ is positive, bounded above by one, but, in contrast to $A(x,y)$,
it is discontinuous at
$(x,y)=(1,1)$.
Indeed, it follows from (\ref{eq:R}) that
$\lim_{\epsilon\downarrow 0} R(1-\epsilon, 1-\epsilon)=1$, but
$\lim_{\epsilon\downarrow 0} R(1, 1-\epsilon)=1/n$.
However, $R(x,y)$ is upper semi--continuous because for every sequence
$\{(x_m, y_m)\}$ such that $(x_m, y_m)\to (1,1)$ as $m\to\infty$
one has
$\limsup_m R(x_m, y_m) \leq R(1,1)=1$. Therefore
(\ref{eq:R-minimax})
is a consequence of the minimax theorem in \citeasnoun{Glicksberg}.

\par
It is worth mentioning  that
a game--theoretic interpretation
of prophet inequalities has been discussed, e.g., in
\citeasnoun{Schmitz} and
\citeasnoun{Meyer-Schmitz}. For independent random variables
\citeasnoun{Schmitz}
studies the question of existence of
minimax strategies  in games against a
prophet  with ratio--type and difference--type payoff functions, while
\citeasnoun{Meyer-Schmitz} focuses on prophet games for martingales and
general stochastic processes.
Corollary~\ref{cor:minimax} deals with iid random variables, and
single threshold stopping rules;
these results complement the ones in the aforementioned papers.

\section{Computation of sharp constants}\label{s-prophet}

 Theorem~\ref{th:R-A} shows that sharp constants in prophet inequalities
 for  single threshold stopping rules
 are determined by the optimal values
 of the two--person zero--sum infinite games on the unit square
 \eqref{eq:R*} and \eqref{eq:A*}. These games can be
 approximated to  any prescribed
 accuracy
 by finite games which are  efficiently solved using simple computational procedures.
 \par
 Specifically, for $R(x,y)$  and $A(x,y)$
 defined in \eqref{eq:R} and \eqref{eq:A} and integer number
 $N$ define $(N-1)\times (N-1)$ matrices $R_N$ and $A_N$ as follows: 
 $R_N :=  \big\{R(\tfrac{i}{N}, \tfrac{j}{N}): i,j= 1, \ldots, N-1\big\}$, and
  $A_N :=  \big\{A(\tfrac{i}{N}, \tfrac{j}{N}): i,j= 1, \ldots, N-1\big\}$,
i.e.,
\begin{align}\label{eq:R-N}
 [R_N]_{ij} &:=
  \left\{\begin{array}{ll}
\frac{1-\big(\frac{i}{N}\big)^{n-1}\frac{j}{N}}{1-\big(\frac{j}{N}\big)^n}, & 1\leq j\leq i\leq N-1,\\[4mm]
\frac{1-(\frac{i}{N})^n}{1-\frac{i}{N}} \frac{1-\frac{j}{N}}{1-\big(\frac{j}{N}\big)^n} , & 1\leq i<j\leq N-1,
                 \end{array}
\right.
\\[4mm]
[A_N]_{ij} &:=\left\{
 \begin{array}{ll}
(1- \tfrac{j}{N}) \sum_{k=1}^{n-1} \big[(\tfrac{j}{N})^k -(\tfrac{i}{N})^k\big], &  1\leq i \leq j \leq N-1, \\[4mm]
\tfrac{j}{N} \big[(\tfrac{i}{N})^{n-1} - (\tfrac{j}{N})^{n-1}\big], &
1\leq j < i\leq N-1.
 \end{array}
\right.
\label{eq:A-N}
\end{align}
Consider the associated finite matrix games
\begin{align}
\cR_{n, N}^* := \min_\mu \max_\lambda  \Big\{ \lambda^TR_N\mu: \;
\lambda^T {\bf 1} =1,\;
 \mu^T {\bf 1}=1,\; \lambda\geq 0, \;\mu\geq 0\Big\},
\label{eq:matrix-R}
\\[2mm]
 \cA_{n, N}^* := \max_\mu \min_\lambda  \Big\{ \lambda^TA_N\mu: \;
\lambda^T {\bf 1} =1,\;
 \mu^T {\bf 1}=1,\; \lambda\geq 0, \;\mu\geq 0\Big\},
\label{eq:matrix-A}
\end{align}
where ${\bf 1}$ stands for the vector of ones, and  $\lambda$ and $\mu$ correspond to
the stopping strategy and the least favorable distribution respectively.
\par
Our current goal is to establish bounds on the
accuracy of approximation of optimal values
of the inifnite games in \eqref{eq:R*} and \eqref{eq:A*},
$\cR_n^*(\sT_{\rm s, r}; \cF_{[0,\infty)})$ and
$\cA_n^*(\sT_{\rm s, r}; \cF_{[0,1]})$,
by
the matrix games $\cR_{n,N}^*$ and  $\cA_{n, N}^*$ respectively.
First, we show that the optimal values $\cR_{n,N}^*$ and $\cA_{n, N}^*$ of
matrix games \eqref{eq:matrix-R} and \eqref{eq:matrix-A} admit interpretations
as the rewards of optimal single threshold stopping rules
for discrete
distributions of special type.
\begin{theorem}\label{th:DN} {\rm \textbf{ (Discrete approximations) } }
Consider
the family of all discrete distribution functions
 with at most $N$ atoms and probabilities taking values
 in the set $\{i/N: i=1, \ldots, N\}$:
\[
 \cD_N:=\bigg\{F: F(x)=\frac{1}{N}\sum_{i=1}^N
 I\{u_i \leq x\}, \;\;u=(u_1, \ldots, u_N)\in \cK_+^N\bigg\},
\]
where $\cK_+^N=\{x\in \bR^N: 0\leq x_1\leq \cdots \leq x_N\}$.
Then
$\cR_{n, N}^*=\cR_n^*\big(\sT_{\rm s, r};  \cD_N\big)$,
$\cA_{n, N}^*=\cA_n^*\big(\sT_{\rm s, r}; \cD_N\big)$.
\end{theorem}
Theorem~\ref{th:DN} shows that the optimal values of
the finite matrix games in
\eqref{eq:matrix-R} and \eqref{eq:matrix-A} provide sharp
constants in the prophet inequalities for stopping rules in $\sT_{\rm s,r}$ and
the class of distributions $\cD_N$.
It is worth mentioning that
for fixed $n$ and any $N_1$ and $N_2$ such that
$N_1>N_2$ and $\cD_{N_1}\subset \cD_{N_2}$ we have that
$\cR_{n, N_1}^* \leq \cR^*_{n, N_2}$ and $\cA^*_{n, N_1}\geq \cA^*_{n, N_2}$.
The proof of Theorem~\ref{th:DN} also demonstrates that the least favorable distributions
from $\cD_N$ are fully determined by $N-1$ numbers $v_j$, $j=1, \ldots, N-1$ which are differences between the subsequent
$(i/N)$--quantiles of $F\in \cD_N$, $v_i:=u_{i+1}-u_i$, $i=1,\ldots, N-1$.
\par
The next statement provides
bounds on
$\cR_{n}^*(\sT_{\rm s, r}; \cF_{[0,\infty)})$ and
$\cA_{n}^*(\sT_{\rm s, r}; \cF_{[0,1]})$
in terms of the optimal values of the
matrix games \eqref{eq:matrix-R} and \eqref{eq:matrix-A}.
\begin{theorem}\label{th:approximation} {\rm \textbf{ (Bounds for discrete approximations) } }
\begin{itemize}
 \item[{\rm (i)}]  Let
 $\Delta_A:= \frac{n-1}{2N}$; then
 for any $n\geq 2$ one has
 \begin{equation}\label{eq:A-2-sided}
  \cA_{n, N}^*  \;\leq\;
  \cA_n^*\big(\sT_{\rm s, r}; \cF_{[0,1]}\big) \;\leq \; \cA_{n,N}^*
  + \Delta_A.
 \end{equation}
 \item[{\rm (ii)}]
 Let $\Delta_R:= \frac{n-1}{2N} [(1-e^{-1})^2-\frac{1}{n-1}]^{-1}$; then
 for any $n\geq 4$ one has
\begin{equation}\label{eq:R-2-sided}
 \cR_{n, N}^* - \Delta_R
 \leq \cR^*_n\big(\sT_{\rm s,r}; \cF_{[0,\infty)} \big)  \leq  \cR_{n, N}^*.
\end{equation}
\end{itemize}
\end{theorem}
The theorem shows that sharp non--asymptotic
constants in prophet inequalities for single threshold rules
in $\sT_{\rm s,r}$
can be computed
with any prescribed accuracy by solution of the finite games in
\eqref{eq:matrix-R}--\eqref{eq:matrix-A}.
For any horizon $n$, choosing sufficiently large $N$,
we can achieve desired precision in computation of the sharp constants.
In particular, the bounds of Theorem~\ref{th:approximation}
show that in order to achieve fixed precision $\epsilon$, the discretization parameter $N$ should grow as $n$ increases.
\par
Note also that the left inequality in \eqref{eq:A-2-sided}
and right inequality in \eqref{eq:R-2-sided} follow straightforwardly from
Theorem~\ref{th:DN}.
According to this theorem
the developed discrete approximation corresponds to the
single threshold
stopping rules in $\sT_{\rm s, r}$. 
It is useful to note that the finite matrix games
\eqref{eq:matrix-R}--\eqref{eq:matrix-A}, that yield the bounds on the optimal constants in \eqref{eq:A-2-sided} and \eqref{eq:R-2-sided}, can be expressed as 
linear programs which are efficiently solved using standard computational tools.
We provide the details below.
\subsection{Ratio--type prophet inequalities}
To compute the sharp constants in ratio--type prophet inequalities
we solve the game \eqref{eq:matrix-R} as follows.
Let $r_i^T$ denote the $i$th row of matrix $R_N$ 
defined in \eqref{eq:R-N}, i.e., $R_N=[r_{1}^T;
r_{2}^T;\cdots;r_{N-1}^T]$. Then
the problem \eqref{eq:matrix-R}
is equivalent~to
\[
 \begin{array}{ll}
  {\displaystyle \min_\mu}  & {\displaystyle  \max_{i=1, \ldots, N-1}\; r_i^T\mu}
  \\*[2mm]
  {\rm s.t.} & \mu^T{\bf 1}=1
  \\*[1mm]
   & \mu \geq 0
 \end{array}
\]
which  is directly cast as a linear optimization problem
\begin{equation} \label{eq:R-opt}
\tag{R}
\begin{array}{ll}
  {\displaystyle \min_{\mu, t}} & t\\[1mm]
  {\rm s.t.} & r_i^T\mu \leq t,\;\;\;i=1, \ldots, N-1\\[1mm]
             & \mu^T {\bf 1}  = 1\\[1mm]
             & \mu\geq 0.
 \end{array}
\end{equation}
If $(\mu^*, t^*)$ is the optimal solution of
\eqref{eq:R-opt}
then the value of the matrix game \eqref{eq:matrix-R} is
$\cR_{n, N}^*=t^*$, and a least favorable distribution
has at most $N$ atoms with $(i/N)$--quantiles $\{u^*_i, i=1, \ldots, N\}$
which are determined by the relationship
\[
u_1^*=0,\;\;\;\;u_{i+1}^*=\sum_{j=1}^i \frac{\mu_j^*}{1-(\frac{j}{N})^n},\;\;\;i=1,\ldots, N-1.
\]
Then the optimal stopping rule 
is associated with the threshold $\theta^*=u_{i^*+1}$, where $i^*$ satisfies
$r_{i_*}^T \mu^*=t^*$.
\par
We solve problem \eqref{eq:R-opt} for different values of  horizon
$n\in \{10, 25, 50, 100, 300, 500,
1000, 2000, 3000\}$ with fixed discretization parameter,
$N=13500$\footnote{This choice of $N$ is dictated by the
computer power limitations. The computations were performed on a
laptop with 32GB RAM and i7 11th generation processor.},
using {\tt Mosek} optimization solver \citeasnoun{Mosek}.
The results
are reported in the second and third columns of Table~\ref{tab:1}. The second
column presents the values $\cR_{n, N}^*$, while
the third column gives  the lower bound $\underline{\cR}_{n, N}^*$
which is the maximum of the lower bound of Theorem~\ref{th:approximation}
and $1- (1-\frac{1}{n})^n$, given by the prophet inequality  \eqref{eq:cor-11}:
\begin{align*}
 &\underline{\cR}_{n, N}^* = \max\Big\{ 1- \Big(1-\frac{1}{n}\Big)^n,\;\;
 \cR_{n, N}^* - \Delta_R\Big\}.
\end{align*}
\par
Accuracy of the presented bounds deteriorates as $n$ grows;
computation of
more accurate bounds requires selecting  much bigger values of $N$.
For instance, in order to compute the sharp constant in the prophet inequality
for $n=100$ with accuracy $10^{-3}$, it is required to have
$N > 1.25\cdot 10^5$,
which leads to a linear program with non--sparse matrices
of dimensionality of several hundred thousands.
\par
From the numbers presented in the second and third columns of Table~\ref{tab:1}
some very useful conclusions on sharp constants can be reliably drawn.
For instance, if $n=25$
then 
$
0.651 \leq \cR^*_{25} \big(\sT_{\rm s, r}; \cF_{[0,\infty)}\big)
\leq 0.654$.

\begin{table}
\begin{center}
 \begin{tabular}{|l||l|c||l|c|}  \hline
 &&&&
 \\[-3mm]
  $n$ & $\cR_{n, N}^*$ & $ \underline{\cR}_{n,N}^*$ & $\cA_{n, N}^*$
  & $\cA_{n,N}^*+\Delta_A$\\ \hline
  10 & 0.6698  &  0.669 & 0.1395 &  0.140 \\
  25 & 0.6540  &  0.651
  & 0.1572 & 0.158\\
  50 & 0.6458 & 0.641
  &0.1644 &  0.166\\
  75 &  0.6427&  0.636
  & 0.1671 & 0.170\\
  100 & 0.6411 &  0.634
  &0.1699 & 0.172\\
  200& 0.6392  &  0.633
  & 0.1708 & 0.178\\
  500 & 0.6409  &  0.632
  & 0.1723& 0.191\\
  1000  &0.6468& 0.632
  &0.1729 & 0.211\\
  2000 &0.6595 &  0.632
  &0.1733 & 0.250\\
  3000 & 0.6726  & 0.632
  &0.1731 & 0.288
  \\
 \hline
 \end{tabular}
\end{center}
\caption{Optimal values $\cR^*_{n, N}$ and $\cA_{n, N}^*$
of problems (R) and (A) and bounds on $\cR^*_n(\sT_{\rm s, r};\cF_{[0,\infty)})$ and $\cA_{n}^*(\sT_{\rm s,r}; \cF_{[0,1]})$
for different values of $n$, where $N=13500$ for problem \eqref{eq:R-opt}
and $N=13000$ for problem \eqref{eq:A-opt}.\label{tab:1}}
\end{table}
\subsection{Difference-type prophet inequalities}
The sharp constants for the difference--type prophet inequality
are obtained as solution of the finite matrix game \eqref{eq:matrix-A}
with payoff matrix $A_N$ defined in \eqref{eq:A-N}.
If $a_i^T$ denotes the $i$th row of matrix $A_N$, i.e.,
$A_N=[a_{1}^T; a_{2}^T;\ldots;a_{N-1}^T]$, then \eqref{eq:matrix-A} is equivalent to
the  linear program
\begin{equation}\label{eq:A-opt}
 \tag{A}
\begin{array}{ll}
 {\displaystyle \max_{\mu, t}}  & \;\;\;\;\;t
 \\[1mm]
 {\rm s.t.} & \;\;\;\;\;a_i^T\mu \geq t,\;\;\;i=1, \ldots, N-1
 \\[1mm]
 &\;\;\;\;\;\mu^T{\bf 1}= 1
 \\[1mm]
 & \;\;\;\;\;\mu\geq 0.
\end{array}
\end{equation}
If $(\mu^*, t^*)$ is the optimal solution of \eqref{eq:A-opt} then
$\cA_{n, N}^*=t^*$, and the least favorable distribution $F_*$ is the discrete
distribution with $i/N$--quantiles $u^*_i$ determined by
$u_1^*=0$, $u^*_{i+1} = \sum_{j=1}^i \mu_j^*$, $i=1,\ldots, N-1$.
\par
The problem \eqref{eq:A-opt} is solved for
$n\in \{10, 25, 50, 100, 300, 500, 1000, 2000, 3000\}$, and $N=13000$. The last two columns in Table~\ref{tab:1} contain the optimal value $\cA_{n, N}^*$,  and
$\cA_{n, N}^*+ \Delta_A$,
which are the lower and upper bounds on the sharp constant
$\cA_n^*\big(\sT_{\rm s, r}; \cF_{[0,1]}\big)$ in the difference--type prophet inequality.
To the best of our knowledge, the existing literature does not contain prophet inequalities for
single threshold stopping rules in the iid setting. However, the numbers
in the last two columns on Table~\ref{tab:1} can be
compared to the bounds on the values
$\cA_n^*(\sT_{\rm all}; \cF_{[0,1]})$
reported in \citeasnoun{Hill-Kerz-1} for the class of
all possible stopping rules:
$\cA_n^*(\sT_{\rm all}; \cF_{[0,1]}) \leq b_n$ with $b_{10,000}\approx 0.111$.
The inferiority the optimal single threshold stopping rule
in comparison with the optimal stopping rule is rather moderate: e.g., for $n=200$ the regret
increases from $0.11$ to $0.17$.

\section{Prophet inequalities for restricted families of distributions}
\label{s-num}

The proofs of Theorems~\ref{th:R-A} and~\ref{th:DN}
heavily exploit the fact that both
the expected value of the maximum, $M_n(F)=\bE\max_{1\leq t\leq n} X_t$, and
the reward $V_n[\tau_{p_x}(\theta_x); F]$ of a single threshold stopping rule with associated
parameters $\theta_x$ and $p_x$,
are linear functionals of the quantile function
$F^\leftarrow$. In particular, we show that
\[
V_n(\tau_{p_x}(\theta_x); F)  =
\int_{0}^1 \Big[(1-x^{n-1})\min\big\{1, \tfrac{1-y}{1-x}\big\} +
  x^{n-1}
  (1-y)\Big] \rd F^\leftarrow(y),
\]
and
\[
 M_n(F)= \bE \max_{1\leq t\leq n} X_t = \int_{0}^{1} (1-y^n) \rd F^\leftarrow(y).
\]
\par
Another step in our derivation  of the sharp constants is the approximation of $F$ by the class of discrete distributions  $\cD_N$.
This allows one to reduce infinite dimensional optimization problems to finite--dimensional ones and
to efficiently solve them  on a computer.
It is worth noting that
any distribution function $F$ can be approximated in the $\bL_\infty$--norm
by a function
from $\cD_N$ with accuracy $1/N$. Since the value  of $N$
can be chosen (at least, theoretically)
as large as wished,  the sharp constants
in prophet inequalities can be computed to any prescribed level of accuracy.
\par
Specifically,
the proof of Theorem~\ref{th:DN} demonstrates
that for $F\in\cD_N$ one has
\[
 M_n(F)= d^T v,\;\;d=[d_1; \ldots; d_{N-1}],\;\;d_j = 1- \big(\tfrac{j}{N}\big)^n,\;\;j= 1, \ldots, N-1,
\]
and $V_n(\tau_{p_i}(u_i); F) = b_i^T v$, where $p_i:= (F(u_i)-i/N)(F(u_i)-F(u_i-))$,
 $b_i=[b_{i,1};\ldots;b_{i,N-1}]$,
\[
 b_{i,j}=\left\{\begin{array}{ll}
1-\big(\frac{i}{N}\big)^{n-1}\frac{j}{N}, & 1\leq j\leq i\leq N-1,\\[2mm]
\frac{1-(\frac{i}{N})^n}{1-\frac{i}{N}} (1-\frac{j}{N}), & 1\leq i<j\leq N-1,
                 \end{array}
\right.
\]
and  $v=[v_1;\ldots,v_{N-1}]$, $u_1=0$, $v_j=u_{j+1}-u_{j}$, $j=1, \ldots, N-1$ with
$u_j$'s being the $(j/N)$--quantiles of $F$.
Then the derivation of the ratio--type and different--type
prophet inequalities
are formulated
as optimization problems
\begin{equation}\label{eq:problem-R1}
 \begin{array}{cl}
 {\displaystyle \min_v} & \;\;\;\;{\displaystyle
 \max_{i=1, \ldots, N-1} \;\;\frac{b_i^T v}{d^Tv}}
 \\[4mm]
 {\rm s.t.}
 & \;\;\;\;\;\;\;\;\;v\geq 0.
 \end{array}
 \end{equation}
 and
 \begin{equation}
 \label{eq:problem-A1}
\begin{array}{cl}
 {\displaystyle \max_v} & \;\;\;\;
 {\displaystyle \min_{i=1, \ldots, N-1}}\;\; [d^Tv - b_i^T v]
 \\[4mm]
 {\rm s.t.} & \;\;\;\;{\bf 1}^T v =1
 \\
 & \;\;\;\;v\geq 0.
 \end{array}
 \end{equation}
 It is readily seen that \eqref{eq:problem-R1} and \eqref{eq:problem-A1} are equivalent to
 \eqref{eq:R-opt} and \eqref{eq:A-opt} respectively.
These problems are reduced to
linear programs that are efficiently solved on a computer.
The corresponding optimal values provide
approximate sharp constants with accuracy guarantees  given
in Theorem~\ref{th:approximation}.
\par
What is perhaps more important is that
the proposed approach can be used in order
to compute approximate sharp prophet inequalities for restricted families
of distributions, such as distributions with bounded variance, second moment, or
other constraints on the tail behavior. In fact, prophet inequalities under
any condition on the quantile function that results
in
convex constraints  in
\eqref{eq:problem-R1}--\eqref{eq:problem-A1}
can be efficiently computed.
We illustrate this fact in the following two examples.
\paragraph{Difference--type  prophet inequality for distributions with bounded variance.}
Let $\cF_\sigma$ be the class of  distribution functions on $[0,\infty)$
with variance bounded by constant $\sigma^2<\infty$,
\[
\cF_\sigma:=\Big\{F\in \cF_{[0,\infty)}:
{\rm var}(X_i) = \int_{(0,1)}\int_{(0,1)} \big[ x\wedge y -xy\big] \rd F^{\leftarrow}(x)\rd F^{\leftarrow}(y) \leq \sigma^2\Big\}.
\]
If $F\in \cF_\sigma \cap \cD_N$ then the condition
${\rm var}(X_i)\leq \sigma^2$ is equivalent to
\[
 \sum_{i=1}^{N-1}\sum_{j=1}^{N-1} \Big(\frac{i}{N}\wedge \frac{j}{N} -
 \frac{ij}{N^2}\Big) (u_{i+1}-u_{i})(u_{j+1}-u_{j}) = v^T Qv,
\]
where, $v=[v_1;\ldots;v_{N-1}]$ with $v_i=u_{i+1}-u_{i}$, and
$Q_{ij}= \frac{i \wedge j}{N} - \frac{ij}{N^2}$, $i,j=1, \ldots, N-1$.
\par
The problem \eqref{eq:problem-R1} associated with the ratio--type
prophet inequality
is scale invariant in the sense that
its optimal solution is defined up to a scale parameter. Therefore the optimal
value  does not change
when an upper bound on variance is added.
 The same is true when any one--sided linear constraint on $v$ is imposed.
Therefore we discuss only the difference--type prophet
inequality for the family
$\cF_\sigma\cap \cD_N$.
\par
In this situation
in
the problem \eqref{eq:problem-A1}
the constraint ${\bf 1}^Tv=1$
should be  replaced
by the quadratic constraint $v^TQv\leq \sigma^2$:
\begin{equation*}
\begin{array}{ll}
 {\displaystyle \max_v} & \;\;\;\;
 {\displaystyle \min_{i=1, \ldots, N-1}} [d^Tv - b_i^T v]
 \\[4mm]
 {\rm s.t.} & \;\;\;\;v^TQv \leq \sigma^2
 \\
 & \;\;\;\;v\geq 0.
 \end{array}
 \end{equation*}
The optimal value of this problem is
$\cA_n^*(\sT_{\rm s}; \cF_\sigma\cap
\cD_N) = \kappa_n\sigma$, where $\kappa_n$ is given by
\begin{align}
  \kappa_n:=
  \max_{z,t}\Big\{t: ({\bf 1} d^T - B)z - t{\bf 1}\geq 0,\;
 \|Q^{1/2} z\|_2 \leq 1,\;z\geq 0,\;t\geq 0\Big\},
 \label{eq:cn}
\end{align}
where $B=\{b_{i,j}\}$, $i,j=1, \ldots, N-1$.
Table~\ref{tab:2} displays  values of $\kappa_n$ for
$n\in \{10, 25, 50, 75, 100, 250, 500, 750, 1000\}$
obtained by solving optimization problem in \eqref{eq:cn}
for $N=7000$; we used {\tt CVX} package for specifying and solving convex programs
\citeasnoun{CVX}
together
with the {\tt Mosek} optimization  solver.
\par
It is seen that the sequence $\{\kappa_n\}$ increases with $n$.
The growth of $\{\kappa_n\}$ can be compared with results of
\citeasnoun{Kennedy-Kertz}, who
considered the setting of independent random variables, $F^{(n)}=\prod_{i=1}^nF_i$,  with marginal distributions $F_i$'s having bounded variance,
${\rm var}\{X_i\}\leq \sigma^2$, $i=1, \ldots, n$.
They
established
an upper bound on the worst--case
regret of the optimal stopping rules; specifically
\[
 \cA_n^*(\sT_{\rm all}; \cF_\sigma^{(n)}) = \sup_{F^{(n)}\in \cF_\sigma^{(n)}}
 \cA_n(\sT_{\rm all}; F^{(n)}) \leq c_n \sigma \sqrt{n-1},
\]
where $c_n\leq 1/2$, $\liminf_n c_n\geq \sqrt{\ln 2- 1/2}\approx 0.439$.
Since the setting of the iid random variables is a specific case,
we also have $\cA^*_n(\sT_{\rm all}; \cF_\sigma)\leq c_n \sigma \sqrt{n-1}$.
Note that  for large $N$, the value
$\cA^*_n(\sT_{\rm s, r}; \cF_\sigma\cap \cD_N)$ approximates
$\cA^*_n(\sT_{\rm s, r}; \cF_\sigma)$. For the data in
Table~\ref{tab:2} we have $\max_{n}\{\kappa_n/c_n\sqrt{n-1}\} \leq 0.451$, where
the maximum is taken over $n\in\{10, 25, 50, 75, 100, 250, 500, 750, 1000\}$.
Thus, the upper bound of \citeasnoun{Kennedy-Kertz} is at least twice higher than
the actual worst--case regret of the optimal single threshold
stopping rule on the class $\cF_\sigma\cap \cD_N$ with the values of $n$ and $N$ indicated above. Note, however, that  \citeasnoun{Kennedy-Kertz}
consider more general setting of independent random variables, and
do not make statements on sharpness of the derived prophet inequality.

\begin{table}
\begin{center}
 \begin{tabular}{|c|ccccccccc|} \hline
$n$ & 10& 25 & 50 & 75 &100 &250 & 500 & 750& 1000 \\ \hline
$\kappa_n$ &0.594 & 0.957 & 1.361 & 1.670 & 1.930&3.056 & 4.319& 5.280& 6.082 \\
\hline
 \end{tabular}
 \end{center}
 \caption{The values of $\kappa_n$ in \eqref{eq:cn} as a function of $n$
 \label{tab:2}.}
\end{table}

\paragraph{Ratio--type prophet inequality for Pareto--like distributions.}
Let $1<p_1<p_0$ be real numbers and consider the family of distributions
\[
 \cF(p_0,p_1)=\Big\{F\in \cF_{[1,\infty)}: 1-x^{-p_1}\leq F(x)\leq 1-x^{-p_0},\;\;
 \forall x\in [1, \infty)\Big\}.
\]
If $F\in \cF(p_0,p_1)$ then $(1-t)^{-1/p_0}\leq F^\leftarrow(t)\leq (1-t)^{-1/p_1}$,
and for $F\in \cF (p_0,p_1)\cap \cD_N$ we have
\begin{equation}\label{eq:condition-Pareto}
   q_{1,i}:=\Big(\frac{N}{N-i}\Big)^{1/p_0} \leq u_i\leq \Big(\frac{N}{N-i}\Big)^{1/p_1}=:q_{0,i},\;\;\;i=1,\ldots, N-1.
\end{equation}
The condition in the definition of $\cF(p_0,p_1)$
imposes restrictions on the distribution tail. Larger values of $p_1$ result
in lighter distribution tails, and it is expected that in such situation
the worst--case competitive ratio of the single threshold stopping
rules will be closer to one.
\par
The optimization problem associated with the family $\cF(p_0,p_1)\cap \cD_N$ takes the
form
\begin{equation}\label{eq:R1-pareto}
 \begin{array}{ll}
 {\displaystyle \min_v} & \;\;\;\;{\displaystyle
 \max_{i=1, \ldots, N-1} \;\;\frac{b_i^T v}{d^Tv}}
 \\[4mm]
 {\rm s.t.}
 & \;\;\;\;\;\;\;\;\;q_1 \leq Qv \leq q_0
 \\
 & \;\;\;\;\;\;\;\;\;v\geq 0,
 \end{array}
 \end{equation}
where $Q$ is the lower triangular
matrix with all entries equal to one, and
$q_0=[q_{0,1};\ldots; q_{0, N-1}]$ and
$q_1=[q_{1,1};\ldots; q_{1, N-1}]$; see \eqref{eq:condition-Pareto}.
The optimal values
$\cR^*_{n, N}(p_0,p_1):=\cR_n^*(\sT_{\rm s}; \cF(p_0,p_1)\cap \cD_N)$
of \eqref{eq:R1-pareto} for $p_1=5$, $p_0=20$, $n\in \{10,25,50,75,100,250,
500,750,1000\}$ computed with $N=7000$ are presented in Table~\ref{tab:3}.
These values can be compared with values of $\cR_{n, N}^*$
in Table~\ref{tab:1}. As expected, $\cR^*_{n, N}(p_0,p_1)>\cR_{n, N}^*$
because
the family of considered distributions is narrower.
Note  also that the values of $\cR^*_{n, N}(p_0,p_1)$
decrease as $n$ increases.
\begin{table}
\begin{center}
 \begin{tabular}{|c|ccccccccc|} \hline
$n$ & 10& 25 & 50 & 75 &100 &250 & 500 & 750& 1000 \\ \hline
 $\cR^*_{n, N}(p_0,p_1)$
&  0.897 & 0.865   &  0.846  & 0.837
 &0.831 & 0.815 &0.806 & 0.802 & 0.799   \\
\hline
 \end{tabular}
 \end{center}
 \caption{The optimal
 values $\cR_n^*(p_0,p_1)$ of \eqref{eq:R1-pareto} as a function of $n$
 for $p_1=5$, $p_0=20$ and $N=7000$.
 \label{tab:3}}
\end{table}
\par
In general, sharp constants in ratio--type and difference--type prophet
inequalities for single threshold rules can be efficiently computed
for a variety of different families of distributions.

\section{Proofs of main results} \label{sec:proofs}
\subsection{Proof of Theorem~\ref{th:prophet}}\label{sec:proof-of-Th1}
Let $\theta\geq  0$ be a fixed real number, $p\in [0,1]$,
and  consider  the  stopping rule
$\tau_p(\theta)$ defined in (\ref{eq:stop-random}). In the subsequent proof by convention we set  $\prod_{j=1}^0  =1$.
\par
 We have
 \allowdisplaybreaks
 \begin{align*}
  \bE X_{\tau_p(\theta)} & = \bE \sum_{t=1}^{n-1} X_t \big[ {\bf 1}(X_t>\theta)+
  {\bf 1}(X_t=\theta, \xi_t=1)\big] \prod_{j=1}^{t-1}
    \big[ {\bf 1}(X_j<\theta)+
  {\bf 1}(X_j=\theta, \xi_j=0)\big]
  \\
 &\;\;\;\hspace{65mm} +\; \bE X_n \prod_{j=1}^{n-1}
    \big[ {\bf 1}(X_j<\theta)+
  {\bf 1}(X_j=\theta, \xi_j=0)\big]
  \\
  &= \big[\bE X_t {\bf 1}(X_t>\theta) +
  \theta p \Delta(\theta)
  \big]
  \sum_{t=1}^{n-1}\big[ F(\theta-)+ (1-p)\Delta(\theta)\big]^{t-1}
  +
  \bE X_n \big[F(\theta-)+ (1-p)\Delta(\theta) \big]^{n-1}
  \\
  &=\big[\bE X_t {\bf 1}(X_t>\theta) +
  \theta p \Delta(\theta)\big] \frac{1- F_p^{n-1}(\theta)}{1-F_p(\theta)}
  +  F_p^{n-1}(\theta) \bE X_n
  \\[2mm]
  &= \big[\theta(1-F_p(\theta)) + \int_\theta^\infty [1-F(x)]\rd x \big]\frac{1- F_p^{n-1}(\theta)}{1-F_p(\theta)}
  +  F_p^{n-1}(\theta) \bE X_n,
 \end{align*}
and  \eqref{eq:iid-1} follows because $\int_{\theta}^\infty [1-F(x)]\rd x=
\int_{\theta}^\infty
[1-F_p(x)]\rd x$. Thus, \eqref{eq:iid-1} is proved.
\par
Furthermore,
 \begin{align*}
  \bE X_{\tau_p(\theta)}  &= \bE \sum_{t=1}^{n} X_t \big[ {\bf 1}(X_t>\theta)+
  {\bf 1}(X_t=\theta, \xi_t=1)\big] \prod_{j=1}^{t-1}
    \big[ {\bf 1}(X_j<\theta)+
  {\bf 1}(X_j=\theta, \xi_j=0)\big]
  \\
  &\;\;\;+\; \bE X_n \prod_{j=1}^{n}
    \big[ {\bf 1}(X_j<\theta)+
  {\bf 1}(X_j=\theta, \xi_j=0)\big]
  \\
&= \big[\bE X_t {\bf 1}(X_t>\theta)  + \theta p \Delta(\theta)\big]
  \sum_{t=1}^n [F(\theta-) + (1-p)\Delta(\theta)]^{t-1}
\\
 &\;\;\; + \big[\bE X_n {\bf 1}\{X_n<\theta\} + \theta
  (1-p)\Delta(\theta)\big][F(\theta-) + (1-p)\Delta(\theta)]^{n-1}
 \\[2mm]
  &=\Big\{\theta(1-F_p(\theta)) + \int_\theta^\infty [1-F(x)]\rd x \Big\}\frac{1- F_p^{n}(\theta)}{1-F_p(\theta)}
  + F_p^{n-1}(\theta) \Big\{\int_{[0,\theta]} x\rd F(x) -p\theta \Delta(\theta)\Big\}.
 \end{align*}
This completes the proof of \eqref{eq:iid}.

\subsection{Proof of Theorem~\ref{th:R-A}}
Let $F\in \cF_{[0,\infty)}$ be a fixed distribution, and for fixed $\theta\geq 0$ and $p\in [0,1]$ let
$F_p(\theta)$ be defined in~\eqref{eq:Fp}.
For any $x\in [0,1]$ there exists
pair $(\theta_x, p_x)\in [0,\infty]\times [0,1]$ such that $F_{p_x}(\theta_x)=x$. Indeed, by definition of $F_p(\theta)$,
\begin{itemize}
\item[(a)] if $\theta_x:= F^{\leftarrow}(x)$ and $F(\theta_x)=x$ then for any $p_x\in [0,1]$ one has 
$F_{p_x}(\theta_x)=x$;
\item[(b)] if $\theta_x:=F^{\leftarrow}(x)$ and $F(\theta_x)>x$ then for 
$p_x:= (F(\theta_x)-x)/(F(\theta_x)-F(\theta_x-))$ one has $F_{p_x}(\theta_x)=x$. 
\end{itemize}
 Then,
 according to  \eqref{eq:iid-1}, the reward of $\tau_{p_x}(\theta_x)$ is given~by
 \begin{align}
  V_n(\tau_{p_x}(\theta_x); F)&= (1-x^{n-1}) \Big[F^\leftarrow(x)+\tfrac{1}{1-x}\int_{F^\leftarrow(x)}^\infty [1-F(t)]\rd t\Big] + x^{n-1}\int_0^\infty [1-F(x)]\rd x
  \nonumber
  \\
  &= (1-x^{n-1}) \Big[F^\leftarrow(x)+\tfrac{1}{1-x} \int_{x}^1 (1-y) \rd F^\leftarrow(y)\Big]+
  x^{n-1}\int_{0}^1 (1-y)\rd F^\leftarrow(y)
 \nonumber
 \\
  & = \int_{0}^1 \Big[(1-x^{n-1})\min\big\{1, \tfrac{1-y}{1-x}\big\} +
  x^{n-1}
  (1-y)\Big] \rd F^\leftarrow(y).
\label{eq:V-x}
  \end{align}
Furthermore,
\[
 M_n(F)= \bE \max_{1\leq t\leq n} X_t = \int_{0}^{1} (1-y^n) \rd F^\leftarrow(y).
\]
Therefore the constant in the  sharp ratio--type prophet inequality is equal
to the optimal value of the following optimization problem
\begin{equation}\label{eq:problem}
 \cR^*_n (\sT_{\rm s,r}; \cF_{[0,\infty)}) =\inf_{F^\leftarrow}
 \sup_{0\leq x\leq 1}
 \frac{\int_{0}^{1} \Big[(1-x^{n-1})\min\big\{1, \tfrac{1-y}{1-x}\big\} +
  x^{n-1}
  (1-y)\Big] \rd F^\leftarrow(y)}
  {\int_{0}^{1} (1-y^n) \rd F^\leftarrow(y)},
\end{equation}
where infimum is taken over all quantile functions of probability distribuitons on $[0, \infty)$.
The problem is equivalent to
\begin{equation}\label{eq:6.5}
\begin{array}{rl}
 {\displaystyle
 \inf_{F^\leftarrow} \sup_{0\leq x\leq 1}} &\;\;\;\;\;
 {\displaystyle \int_0^1 \Big[(1-x^{n-1})\min\big\{1, \tfrac{1-y}{1-x}\big\} +
  x^{n-1}
  (1-y)\Big] \rd F^\leftarrow(y)}
  \\[4mm]
  {\rm s.t.}
  & {\displaystyle \;\;\;\;\;\int_0^1 (1-y^n) \rd F^\leftarrow(y)=1,}
\end{array}
\end{equation}
\par
Let $\mu$ be the right continuous function such that
$\rd \mu (y):=(1-y^n) \rd F^{\leftarrow}(y)$ for all $y\in [0,1]$.
Then
$\mu$ is a probability distribution on $[0,1]$.
Considering the randomized choice of $x\in [0,1]$ according to 
a distribution $\lambda$ on $[0,1]$,
we observe that
the optimal values of (\ref{eq:problem}) and
(\ref{eq:R*}) are equal.  This proves the first statement of the theorem.
\par
As for
the game--theoretic representation
for the difference--type prophet inequality, we observe that
\[
 \sup_{F^\leftarrow} \inf_{0\leq x\leq 1} [M_n(F)-V_n(\tau_{p_x}(\theta_x); F)] =
 \sup_{F^\leftarrow} \inf_{0\leq x\leq 1} \int_0^1 A(x, y) \rd F^\leftarrow (y),
\]
where the supremum is taken over all quantile functions of
distributions on $[0,1]$. This constraint can be written in the form
$F^{\leftarrow}(1)= \int_0^1 \rd F^{\leftarrow}(t) \leq 1$.
Defining  the right--continuous function $\mu$ such that
$\rd \mu(y)= \rd F^{\leftarrow}(y)$ and using the same reasoning as above, we come to (\ref{eq:A*}).
%
\subsection{Proof of Theorem~\ref{th:DN}}
Assume that $F\in \cD_N$, i.e.,
$F(x)=\frac{1}{N}\sum_{i=1}^N I(u_i\leq x)$
with some $0=u_0\leq u_1\leq \cdots \leq u_N$ for $x\geq 0$. For such distributions
the set of all possible thresholds of stopping rules is restricted
to
the $(i/N)$--quantiles $\{u_i, i=0, \ldots, N\}$ of $F$.
Define
$v_0:= u_1$, $v_j:=u_{j+1}-u_{j}$,  $j=1, \ldots, N-1$.
The quantile function $F^{\leftarrow}(u)$ of $F\in \cD_N$
is given by
\[
F^{\leftarrow}(u)= \sum_{j=0}^{N-1} v_j I\{\tfrac{j}{N}<u\leq \tfrac{j+1}{N}\}.
\]
Therefore
it follows from
\eqref{eq:V-x} that the reward of the stopping rule
$\tau_{p_i}(u_i)$ with threshold $u_i$, $i=0, \ldots, N-1$
and randomization probability $p_i:= (F(u_i)-i/N)(F(u_i)-F(u_i-))$
is
\begin{align}
 V_n(\tau_{p_i}(u_i); F)
  =
 \sum_{j=0}^{N-1} \bigg[\Big(1-\big(\tfrac{i}{N}\big)^{n-1}\Big)
 \min\Big\{1, \tfrac{1-\frac{j}{N}}{1-\frac{i}{N}}\Big\} +& \big(\tfrac{i}{N}\big)^{n-1} \big(1-\tfrac{j}{N}\big)\bigg] v_j
 = \sum_{j=0}^{N-1} R(\tfrac{i}{N}, \tfrac{j}{N})v_j,
 \label{eq:V_n(tau_i)}
 \end{align}
 where $R(\cdot, \cdot)$ is given by \eqref{eq:R}.
  It follows from
 \eqref{eq:V_n(tau_i)} that
 \[
  V_n(\tau_{p_i}(u_i); F)= b_i^T v,\;\;\;i=0, \ldots, N-1,
 \]
where $b_i=[b_{i,0};\ldots;b_{i,(N-1)}]$, $i=0, \ldots, N-1$ are vectors with
entries
 \[
 b_{i,j}= \left\{\begin{array}{ll}
1-\big(\frac{i}{N}\big)^{n-1}\frac{j}{N}, & 0 \leq j\leq i\leq N-1,\\[2mm]
\frac{1-(\frac{i}{N})^n}{1-\frac{i}{N}} (1-\frac{j}{N}), & 0 \leq i<j\leq N-1.
                 \end{array}
\right.
\]
Moreover,  for $F\in \cD_N$ we have
\begin{equation*}
 M_n(F) = \bE\max_{1\leq t\leq n} X_t= \sum_{j=0}^{N-1}\big[1-\big(\tfrac{j}{N}\big)^n\big]v_j =: d^T v,
 \end{equation*}
where $d\in \bR^{N}$,
$d_j= 1-\big(\tfrac{j}{N}\big)^n$, $j=0, \ldots, N-1$.
Then  \eqref{eq:6.5} implies that
\begin{align}\label{eq:pAv}
\cR^*_n (\sT_{\rm s,r}; \cD_N)=
\max_\lambda \min_v\big\{ \lambda^TBv: \lambda^T{\bf 1}=1,\;
d^T v=1,\; \lambda\geq 0, \;v\geq 0\big\},
\end{align}
where $B=[b_0^T; b_1^T;\ldots;b_{N-1}^T]$, and $\lambda$ stands for the probability
vector that
defines the single threshold rule: the threshold $u_i$ is selected with the probability $\lambda_i$, $i=0, \ldots, N-1$.
Observing that 
\[
[b_0^T/d_0; b_1^T/d_1;\cdots;b_{N-1}^T/d_{N-1}] = \big\{R(\tfrac{i}{N}, \tfrac{j}{N}): i,j=0, \ldots, N-1\big\}=: \tilde{R}_N
\]
we obtain that
\eqref{eq:pAv} is equivalent to
\begin{equation}\label{eq:R-tilde}
 \cR^*_n (\sT_{\rm s,r}; \cD_N)=
\max_\lambda \min_\mu\big\{ \lambda^T \tilde{R}_N \mu: \lambda^T{\bf 1}=1,\;
\mu^T{\bf 1}=1,\; \lambda\geq 0, \;\mu\geq 0\big\}.
\end{equation}
To complete the proof that $\cR^*_n (\sT_{\rm s,r}; \cD_N)=\cR_{n, N}^*$ we need 
to show that the optimal value of the above game with $N\times N$ matrix $\tilde{R}_N= \big\{R(\tfrac{i}{N}, \tfrac{j}{N}): i,j=0, \ldots, N-1\big\}$ coincides with the optimal
value of the game with $(N-1)\times (N-1)$ matrix 
$R_N= \big\{R(\tfrac{i}{N}, \tfrac{j}{N}): i,j=1, \ldots, N-1\big\}$.
This fact is a consequence of the following dominance relationships between the rows and columns of matrix $\tilde{R}_N$.
\begin{itemize}
 \item[(i).]  The pure $\lambda$--strategy $[1;0;\cdots;0]$ is  
 dominated by the strategy
 $[0;1;0\ldots;0]$.  Indeed,   
 \begin{align*}
 & R(0, \tfrac{j}{N}) = \frac{1-\tfrac{j}{N}}{1-(\tfrac{j}{N})^n}, \;\;\forall j=0, \ldots, N-1,
 \\
 & R(\tfrac{1}{N},0) = R(\tfrac{1}{N}, \tfrac{1}{N})=1,\;\;
  R(\tfrac{1}{N}, \tfrac{j}{N})=\frac{1-(\tfrac{1}{N})^n}{1-\tfrac{1}{N}}
  \frac{1-\tfrac{j}{N}}{1-(\tfrac{j}{N})^n}, \;\;\;\forall j=2, \ldots, N-1
 \end{align*}
so that $R(0,0)=R(\tfrac{1}{N},0)$ and $R(0,\frac{j}{N})< R(1, \tfrac{j}{N})$ for all $j=1, \ldots, N-1$. Thus, the zeroth
row of matrix $\tilde{R}_N$ may be eliminated. 
\item[(ii).] The pure $\mu$--strategy $[1;0;\cdots; 0]$ is dominated by the strategy 
$[0;1;0;\cdots; 0]$. Indeed, 
\begin{align*}
& R(\tfrac{i}{N}, 0) = 1, \;\;\forall i=0, \ldots, N-1
 \\
& R(0, \tfrac{1}{N})=\frac{1-\frac{1}{N}}{1-(\frac{1}{N})^n},\;\;
 R(\tfrac{i}{N}, \tfrac{1}{N})= \frac{1-(\tfrac{i}{N})^{n-1}\tfrac{1}{N}}{1-(\tfrac{1}{N})^{n}},\;\;\forall i=2, \ldots, N-1.
\end{align*}
Thus, $R(\tfrac{i}{N}, 0)>R(\tfrac{i}{N},\tfrac{1}{N})$ for all $i=0,\ldots, N-1$, 
i.e, the zeroth column of matrix $\tilde{R}_N$ can be eliminated.  
\end{itemize}
The facts (i)--(ii) together with 
\eqref{eq:R-tilde} and the definition of $\cR^*_{n, N}$ in \eqref{eq:matrix-R}
imply that $\cR^*_n (\sT_{\rm s,r}; \cD_N)=\cR_{n, N}^*$.
\par
To prove that
$\cA_{n}^*(\sT_{\rm s,r}; \cD_N)= \cA_{n, N}^*$
we note that with the introduced notation for any $F\in \cD_N$
the regret of the stopping rule $\tau_{p_i}(u_i)$ associated  with the threshold
$u_i$ is
\[
\cA_n(\tau_{p_i}(u_i); F)= M_n(F)-V_n(\tau_{p_i}(u_i); F) = d^Tv - b_i^T v,\;\;\;
i=0, \ldots, N-1.
\]
For $F\in \cF_{[0,1]}$ one has $v^T {\bf 1} \leq 1$;
therefore
\begin{align*}
 \cA_n^*(\sT_{\rm s,r}; \cD_N) = &\max_{v} \min_{i=1, \ldots, N-1} \{ d^Tv - b_i^T v: \;v^T {\bf 1}\leq 1,\;v\geq 0\}\\
  &= \max_{v} \min_{\mu}
 \big\{ \mu^T({\bf 1} d^T-B)v : \mu^T{\bf 1}=1,\;
 v^T {\bf 1} = 1,\;v\geq 0, \;\mu\geq 0\big\}.
\end{align*}
Observe that
\[
 {\bf 1}d^T -B = \big\{A(\tfrac{i}{N}, \tfrac{j}{N}): i,j=0, \ldots, N-1\big\}=: \tilde{A}_N.
\]
Then the  statement of the theorem follows from
the fact that, similarly to the above,  
the zeroth row and zeroth column of $\tilde{A}_N$ can be 
eliminated by dominance considerations. 

\subsection{Proof of Theorem~\ref{th:approximation}}
\paragraph{Proof of statement~(i).}
1$^0$. We begin with a simple lemma.
\begin{lemma}\label{lem:A}
One has
\begin{align*}
 |A(x, y) - A(x^\prime, y)| \leq (n-1)|x-x^\prime|,\;\;\forall x,
 x^\prime, y\in [0,1]
 \\
 |A(x, y) - A(x, y^\prime)| \leq (n-1)|y-y^\prime|,\;\;\forall x,
 y, y^\prime \in [0,1].
\end{align*}
\end{lemma}
\begin{proof} We note that $A(x,y)$ is continuous on $[0,1]^2$.
We have for any $y\in [0,1]$
\begin{eqnarray*}
 \Big|\frac{\partial A(x,y)}{\partial x}\Big| &=& (n-1)y x^{n-2} \leq n-1,
\;\;\;\; \;\;\;\;\;\forall (x,y): x>y\\
 \Big|\frac{\partial A(x,y)}{\partial x}\Big| &=& (1-y)\sum_{k=1}^{n-1}k x^{k-1}
=(1-y) \frac{\sum_{j=0}^{n-2}(j+1) x^j}{\sum_{j=0}^{n-2} x^j} \sum_{j=0}^{n-2} x^j
\\
&&\;\;\;\;\;\;\;\;
\leq \Big(\frac{1-y}{1-x}\Big) \max_{j=0, \ldots, n-2} \{(j+1)\} \leq n-1,
\;\;\;\;\;\;\forall (x,y): x<y.
 \end{eqnarray*}
Therefore for any fixed $y\in [0,1]$ if $x>y$ and $x^\prime>y$ then
$|A(x, y) - A(x^\prime, y)|\leq (n-1)|x-x^\prime|$,
and the same inequality holds for any fixed $y\in [0,1]$ and $x<y$ and $x^\prime<y$. Thus,  for any $y\in [0,1]$
\[
 |A(x, y) - A(x^\prime, y)|\leq (n-1)|x-x^\prime|,\;\;\;\forall x, x^\prime\in [0,1].
\]
Similarly, for any $x\in [0,1]$
\begin{eqnarray*}
 \Big|\frac{\partial A(x,y)}{\partial y}\Big| &=& |x^{n-1}-n y^{n-1}|\leq n-1,\;\;\;
 \forall (x,y): x>y.
\end{eqnarray*}
For $x<y$ we have
\begin{eqnarray*}
 \frac{\partial A(x,y)}{\partial y} &=& \sum_{k=1}^{n-1} k y^{k-1} -
 \sum_{k=1}^{n-1}(k+1)y^k + \sum_{k=1}^{n-1}x^k \leq
 \sum_{k=1}^{n-1} ky^{k-1} (1-y)
 \\
 &\leq& \frac{\sum_{k=1}^{n-1}ky^{k-1}}{\sum_{k=1}^{n-1} y^{k-1}} (1-y) \sum_{k=1}^{n-1} y^{k-1} \leq
 \max_{j=0,\ldots, n-2}\{(j+2) \}\leq n-1.
\end{eqnarray*}
On the other hand, the above expression implies that
\[
 \frac{\partial A(x,y)}{\partial y} = 1-ny^{n-1} + \sum_{k=1}^{n-1} x^k \geq 1-n.
\]
Thus, $|\partial A(x,y)/\partial y|\leq n-1$ which implies that
\[
 |A(x, y) - A(x, y^\prime)|\leq (n-1) |y-y^\prime|,\;\;\;\;\forall
 (x,y): x<y.
\]
This completes the proof of the lemma.
\end{proof}
\par
2$^0$.
For brevity we write $\cA^*_n=\cA_n^*(\sT_{\rm s,r}; \cF_{[0,1]})$, and recall that
$\cA_n^*= \sup_\mu \inf_\lambda \bar{A}(\lambda, \mu)$,
where $\bar{A}(\lambda, \mu)$ is defined in \eqref{eq:bar-R-A}.
Consider the finite matrix game
(\ref{eq:matrix-A}) associated with $A_N$ whose  value is
$\cA^*_{n, N}$.  Assume that
$\lambda^{N}=(\lambda_1^N, \ldots, \lambda^N_{N-1})$
and $\mu^N=(\mu_1^N, \ldots, \mu_{N-1}^N)$
are  the optimal mixed strategies in this game, i.e.,
$\bar{A}(\lambda^N, \mu^N)=\cA_{n, N}^*$.
\par
Let $\delta_y$ denote the degenerate distribution at $y\in [0,1]$.
First, we note that
$\max_y \bar{A}(\lambda^N, \delta_y) \geq \cA_n^*$.
Therefore, there exists point $y$, say $y_*$, such that
$\bar{A}(\lambda^N, \delta_{y_*}) \geq \cA_n^*$.
By Lemma~\ref{lem:A}, there exists index $j_*\in \{\added{0}, 1, \ldots, N-1\}$ such that
\begin{align*}
 \big|\bar{A}(\lambda^N, \delta_{y_*})- \bar{A}(\lambda^N, \delta_{j_*/N})\big|
 \leq \sum_{i=1}^{N-1}\lambda_i^N |A(\tfrac{i}{N}, y_*) - A(\tfrac{i}{N}, \tfrac{j_*}{N})|
 \leq (n-1) |y_*- \tfrac{j_*}{N}|\leq \frac{n-1}{2N}.
\end{align*}
Therefore
\[
 \cA_n^* \leq \bar{A}(\lambda^N, \delta_{y_*})\leq \bar{A}(\lambda^N, \delta_{j_*/N})+ \frac{n-1}{2N} \leq \cA^*_{n, N}+ \frac{n-1}{2N}.
\]
which yields the upper bound in \eqref{eq:A-2-sided}.
\par
Furthermore, note that
if $\delta_x$ is the degenerate distribution at $x\in [0,1]$ then
$\min_x \bar{A}(\delta_x, \mu^N)\leq \cA_n^*$.
Therefore there exists $x_*$ such that $\bar{A}(\delta_{x_*}, \mu^N)\leq \cA_n^*$.
By Lemma~\ref{lem:A}, there exists index $i_*\in \{1, \ldots, N-1\}$ such that
\[
 \big|\bar{A}(\delta_{x_*}, \mu^N)- \bar{A}(\delta_{i_*/N}, \mu^N)\big|
 \leq \sum_{j=1}^{N-1}\mu_j^N \big|A(x_*,\tfrac{j}{N}) - A(\tfrac{i_*}{N},
 \tfrac{j}{N})\big|\leq (n-1)|x_*-\tfrac{i_*}{N}\big| \leq \frac{n-1}{2N}.
\]
Therefore
\[
 \cA_n^*\geq \bar{A}(\delta_{x_*}, \mu^N) \geq \bar{A}(\delta_{i_*/N},
 \mu^N) - \frac{n-1}{2N} \geq \cA_{n, N}^* -  \frac{n-1}{2N}.
\]
which completes the proof of  the statement~(i).

\paragraph{Proof of statement~(ii).}
1$^0$. We begin with an auxiliary lemma.
\begin{lemma}\label{lem:support}
Let $\cR_n^*:=\inf_\mu \sup_\lambda \bar{R}(\lambda, \mu)$
be the value of the game on the unit square with payoff kernel  $R(x,y)$  (see
\eqref{eq:R} and
\eqref{eq:bar-R-A}), and  let $\lambda_*$ and $\mu_*$ be the optimal strategies.
Let
\begin{equation}\label{eq:c_n}
 c_n:=-\ln \Big\{1-(1-e^{-1})^2+\frac{1}{n-1}\Big\};
\end{equation}
if $n\geq 4$ then the interval $[1- c_n/n, 1]$ does not belong to the support of
$\lambda_*$.
\end{lemma}
\begin{proof}
Let
$x_0=1- c/n$ with for some $c \in (0,c_n)$, and
assume to the contrary that
$x_0\in {\rm supp}(\lambda_*)$.
Under this assumption
we have
$\bar{R}(\delta_{x_0}, \mu^*)=\cR_n^*$; see, e.g., \cite[Lemma~2.2.1]{Karlin}.
Now let $y_0=1-\frac{1}{n}$;  since $x_0>y_0$,
it follows from \eqref{eq:R} that
\begin{align*}
\bar{R}(\delta_{x_0}, \delta_{y_0})= R(x_0, y_0)=
 \frac{1-(1-\frac{c}{n})^{n-1}
 (1-\frac{1}{n})}{1-(1-\frac{1}{n})^n} \leq \frac{1}{1-e^{-1}}
 \bigg[1- \Big(1-\frac{c}{n}\Big)^n + \Big(1-\frac{c}{n}\Big)^n\frac{1-c}{n-c}
 \bigg].
\end{align*}
Because
$e^{-c}\geq (1-\tfrac{c}{n})^n \geq e^{-c}\big(1-\tfrac{c^2}{2(n-1)}\big)$
we obtain
\begin{align*}
\bar{R}(\delta_{x_0}, \delta_{y_0}) \leq
\frac{1}{1-e^{-1}}\bigg[ 1- e^{-c}\Big(1-\frac{c^2}{2(n-1)} - \frac{1}{n-c}
\Big)
\bigg]\leq \frac{1}{1-e^{-1}}\bigg[1-e^{-c} + \frac{e^{-c} (c^2+2)}{2(n-1)}\bigg].
\end{align*}
For $c\in (0,1)$ one has $e^{-c}(c^2+2)\leq 2$; therefore
\[
 \bar{R}(\delta_{x_0}, \delta_{y_0}) \leq \frac{1}{1-e^{-1}}\bigg[1-e^{-c}+
 \frac{1}{n-1}\bigg].
\]
This inequality shows that for any pair of numbers $n$ and $c\in (0,1)$
such that
\begin{equation}\label{eq:cc}
 1-e^{-c}+ \frac{1}{n-1} < (1-e^{-1})^2
\end{equation}
we obtain
\[
 \cR^*_n = \bar{R}(\delta_{x_0}, \mu_*)\leq
 \bar{R}(\delta_{x_0}, \delta_{y_0})= R(x_0, y_0)
 < 1-e^{-1}.
\]
In particular, if $n\geq 4$ and
\[
 c_n= -\ln \Big\{1-(1-e^{-1})^2+\frac{1}{n-1}\Big\}\approx
 -\ln\Big\{0.6004 + \frac{1}{n-1}\Big\}
\]
then (\ref{eq:cc}) holds for all $c\in (0,c_n)$.
This, however,  stands in contradiction to inequality \eqref{eq:cor-11}.
Therefore $x_0\not\in {\rm supp}(\lambda_*)$, and the proof is completed.
\end{proof}
2$^0$. The next result is an analogue of Lemma~\ref{lem:A} for function
$R(x,y)$.
\begin{lemma}\label{lem:R}
 For any   $\epsilon\in (0,1)$ one has
 \begin{eqnarray}
  |R(x, y) - R(x^\prime, y)| &\leq& \frac{n-1}{1-(1-\epsilon)^n}|x-x^\prime|,\;\;\;\;
  \forall x, x^\prime \in [0,1-\epsilon],\;\forall y\in [0,1],
  \label{eq:R1}
  \\
  |R(x, y) - R(x, y^\prime)| &\leq& \frac{n-1}{1-(1-\epsilon)^n}|y-y^\prime|,\;\;\;\;
  \forall x\in [0,1-\epsilon],\;\; \forall y, y^\prime \in [0,1].
 \label{eq:R2}
 \end{eqnarray}
\end{lemma}
\begin{proof}  Let $y \in [0,1]$ be fixed. By \eqref{eq:R},
if
$0\leq x < y\leq 1$ then
\begin{align*}
 \frac{\partial R(x,y)}{\partial x}=
 \frac{1-y}{1-y^n} \sum_{j=1}^{n-1} j x^{j-1}= \frac{\sum_{j=0}^{n-2} (j+1)x^j}{\sum_{j=0}^{n-2} x^j}\, \frac{\sum_{j=0}^{n-2} x^j}{\sum_{j=0}^{n-1} y^j}
 \leq \max_{0 \leq j\leq n-2} \{(j+1) \}\leq n-1,
\end{align*}
and for   $0\leq y <x\leq 1-\epsilon $
\begin{align*}
 \Big|\frac{\partial R(x,y)}{\partial x}\Big| = \frac{(n-1)x^{n-2} y}{1-y^n}
 \leq \frac{n-1}{1-(1-\epsilon)^n}.
\end{align*}
In view of the above inequalities, for all $x, x^\prime$ such that $0\leq x<y<x^{\prime}\leq 1-\epsilon$ we have
\begin{align*}
 |R(x, y)- R(x^\prime, y)|\leq |R(x, y) - R(y,y)| + |R(y,y)- R(x^\prime, y)|
 \leq \frac{n-1}{1-(1-\epsilon)^n} |x-x^\prime|.
\end{align*}
The above inequalities imply (\ref{eq:R1}).
\par
Similarly, if $x\in [0,1-\epsilon]$ is fixed, and  $x\leq y\leq 1$ then
\[
 \Big|\frac{\partial R(x,y)}{\partial y}\Big| =\frac{1-x^n}{1-x}
 \frac{\sum_{j=0}^{n-2} (j+1)y^j}{(\sum_{j=0}^{n-1} y^j)^2}
 \leq \frac{\sum_{j=0}^{n-1} x^j}{\sum_{j=0}^{n-1} y^j} \frac{\sum_{j=0}^{n-2} (j+1)y^j}{\sum_{j=0}^{n-1} y^j} \leq
 \max_{0 \leq j\leq n-2} \{(j+1) \}\leq n-1,
\]
and for  $y<x\leq 1-\epsilon$
\begin{align*}
 \frac{\partial R(x,y)}{\partial y} = - \frac{x^{n-1}}{1-y^n} +
 \frac{(1-x^{n-1}y)ny^{n-1}}{(1-y^n)^2}\leq \frac{(n-1)x^{n-1}}{1-y^n} \leq
 \frac{n-1}{1-(1-\epsilon)^n},
\end{align*}
and the same inequality holds for $|\partial R(x,y)/\partial y|$.
Combining these inequalities we obtain (\ref{eq:R2}).
\end{proof}
3$^0$. Now, using Lemmas~\ref{lem:support} and~\ref{lem:R}
we complete the proof of statement~(ii) of Theorem~\ref{th:approximation}.
The proof goes along the lines of the proof of statement~(i), part~2$^0$
with the following minor changes:
the pay-off function $A(x,y)$ is replaced by $R(x,y)$, and the infinite game
is considered on the rectangle $[0, 1- c_n/n]\times [0,1]$.
Lemma~\ref{lem:support} ensures that the optimal values of the game on this set
coincides with the one on the unit square.
\par
Recall that $\cR_n^*=\inf_\mu \sup_\lambda \bar{R}(\lambda, \mu)$, and
$\cR_{n, N}^*= \bar{R}(\lambda^N, \mu^N)$, where
$\lambda^N=(\lambda_1^N, \ldots, \lambda_{N-1}^N)$ and
$\mu^N=(\mu_1^N, \ldots, \mu_{N-1}^N)$ are optimal mixing strategies in the finite matrix game \eqref{eq:matrix-R}.
It follows from Lemma~\ref{lem:support} that
$\lambda^N_i=0$ for $\lceil (1-c_n)N\rceil\leq i\leq N-1$,
where $c_n$ is given in \eqref{eq:c_n}.
The definitions imply that
$\min_y \bar{R}(\lambda^N, \delta_y)\leq \cR_n^*$, i.e., there exists
$y_*\in [0,1]$ such that $\bar{R}(\lambda^N, \delta_{y_*})\leq \cR_n^*$.
By Lemma~\ref{lem:R} applied with $\epsilon=c_n/n$
for some $j_*$
\[
 |\bar{R}(\lambda^N, \delta_{y_*}) - \bar{R}(\lambda^N, \delta_{j_*/N})|
 \leq \frac{n-1}{1-(1- \frac{c_n}{n})^n} |y_*-\tfrac{j_*}{N}| \leq
 \frac{n-1}{2N\big(1-(1- \frac{c_n}{n})^n\big)}
\]
Therefore
\[
 \cR_n^* \geq \bar{R}(\lambda^N, \delta_{y_*})\geq
 \bar{R}(\lambda^N, \delta_{j_*/n}) - \frac{n-1}{2N\big(1-(1- \frac{c_n}{n})^n\big)} \geq
 \cR_{n, N}^* - \frac{n-1}{2N\big(1-(1- \frac{c_n}{n})^n\big)}.
\]
Then the lower bound in \eqref{eq:R-2-sided} follows by substitution of
\eqref{eq:c_n}. The upper bound on $\cR_n^*$ is proved similarly.

\subsubsection*{Acknowledgment}
The research is supported by the BSF grant 2020063.

\bibliographystyle{agsm}

\end{document}